\documentclass[11pt,a4paper]{amsart}
\usepackage{amsfonts}

\newtheorem{theorem}{Theorem}[section]
\theoremstyle{plain}

\newtheorem{corollary}{Corollary}[section]

\newtheorem{definition}{Definition}[section]
\newtheorem{example}{Example}[section]

\newtheorem{lemma}{Lemma}[section]

\newtheorem{proposition}{Proposition}[section]
\newtheorem{remark}{Remark}[section]

\numberwithin{equation}{section}
\input{tcilatex}

\begin{document}
\title[$\mathbb{R}$-group actions on locally compact spaces]{Absorptive
continuous $\mathbb{R}$-group actions on locally compact spaces}
\author{Gabriel NGUETSENG}
\dedicatory{\emph{University of Yaounde I}\\
\emph{\ Department of Mathematics}\\
\emph{\ P.O.Box 812, Yaounde, Cameroon}\\
\emph{\ E-mail: \ nguetsengg@yahoo.fr}\\
\emph{Telephone: 23799737978}}

\begin{abstract}
We introduce the notion of an $\mathbb{R}$-group of which the classical
groups $\mathbb{R}$, $\mathbb{Z}$ and $\mathbb{R}_{+}^{\ast }$ are typical
examples, and we study flows $\left( X,\mathcal{H}\right) $, where $X$ is a
locally compact space and $\mathcal{H}$ is a continuous $\mathbb{R}$-group
action on $X$ with the further property that any compact set is \textit{%
absorbed }(in the ordinary meaning in use in the theory of topological
vector spaces) by any neighbourhood of some characteristic point in $X$
called the center of $\mathcal{H}$. The case where $X$ is a locally compact
abelian group is also considered. We are particularly interested in
discussing the asymptotic properties of $\mathcal{H}$, which is made
possible by proving a deep theorem about the existence of nontrivial $%
\mathcal{H}$-homogeneous positive measures on $X$. Also, a close connection
with homogenization theory is pointed out. It appears that the present paper
lays the foundation of the mathematical framework that is needed to
undertake a systematic study of homogenization problems on manifolds, Lie
groups included.
\end{abstract}

\maketitle

\section{INTRODUCTION}

The study of group actions on sets is diversely interpreted according to how
the former and the latter are each structured. For example classical measure
theoretic ergodic theory is based on the notion of a probability space $%
\left( X,\mathcal{B},m\right) $, i.e., $X$ is a given set, $\mathcal{B}$ a $%
\sigma $-algebra on $X$, and $m$ a probability measure on $\left( X,\mathcal{%
B}\right) $. Here the acting group is generally the additive group $\mathbb{Z%
}$ of integers, and the action consists of the iterates (in a suitable
sense) $T^{n}$ $\left( n\in \mathbb{Z}\right) $ of an $m$-preserving
transformation $T:X\rightarrow X$. Though our purpose is not to discourse on
ergodic theory, attention must be drawn to two historical results that are
knowledged as being the keystone of the theory, namely Birkhoff's individual
theorem and von Neumann mean ergodic theorem. See, e.g., \cite[15, 25]{bib12}
for more details and references on ergodic theory. One could provide the set 
$X$ with a topological structure rather than a measure theoretic one and on
the other hand, consider a topological group $E$ acting continuously on $X$.
This would lead to what is broadly interpreted as topological dynamics.
However, some terminological details seem useful. By an action of a
(multiplicatively written commutative) group $E$ on a set $X$ is understood
here a family $\mathcal{H}=\left( H_{\varepsilon }\right) _{\varepsilon \in
E}$ of permutations of $X$ such that%
\begin{equation*}
H_{\varepsilon }\circ H_{\varepsilon ^{\prime }}=H_{\varepsilon \varepsilon
^{\prime }}\text{ and }H_{\varepsilon =e}\left( x\right) =x\qquad \left(
\varepsilon ,\varepsilon ^{\prime }\in E\text{, }x\in X\right) \text{,}
\end{equation*}%
where $\circ $ and $e$ denote usual composition and the identity of $E$,
respectively. If $E$ is a topological group and $X$ is a topological space,
the action $\mathcal{H}$ is termed continuous if the map $\left( \varepsilon
,x\right) \rightarrow H_{\varepsilon }\left( x\right) $ sends continuously $%
E\times X$ (with the product topology) into $X$, in which case the pair $%
\left( X,\mathcal{H}\right) $ is called a \textit{flow}. The reader
interested in topological dynamics is referred, e.g., to \cite[11]{bib7}.
Assuming now $X$ to be a smooth manifold and $\mathcal{H}$ to be, say a
differentiable action of $\mathbb{R}$ on $X$, one arrives at what is
commonly referred to as differentiable dynamics (see, e.g., \cite[23, 24]%
{bib22}). It is worth mentioning in passing that in both topological and
differentiable dynamics, $X$ is generally assumed to be compact, which
allows a most highly development of the theory.

This study is concerned with absorptive continuous $\mathbb{R}$-group
actions on locally compact spaces. Crudely speaking, calling $E$ an $\mathbb{%
R}$-group if $E=\mathbb{R}$, $\mathbb{Z}$ or $\mathbb{R}_{+}^{\ast }$ (the
multiplicative group of positive reals), we study flows $\left( X,\mathcal{H}%
\right) $\ where $X$ is a locally compact space and $\mathcal{H=}\left(
H_{\varepsilon }\right) _{\varepsilon \in E}$ is a continuous action of an $%
\mathbb{R}$-group $E$ on $X$ with the further property that any compact set
is \textit{absorbed }(in a sense to be specified later) by any neighbourhood
of some characteristic point in $X$ called the center of $\mathcal{H}$. In
view of this setting, the present study should naturally lie within the
scope of topological dynamics and one could be interested in discussing
classical notions such as, e.g., minimal sets and minimal flows,
equicontinuity, proximality, distality, etc., for which we refer to \cite[%
11, 13]{bib2}. The orientation that commands our attention is quite another
matter. In fact, the specificity of this study is due to the deep impact of
the absorptiveness hypothesis in so far as the latter turns out to underlie
a very special setting beyond classical dynamical theories. For example,
whereas invariant positive measures and compact spaces play fundamental
roles in classical topological dynamics, in the present context there exists
no nontrivial invariant positive measure, and on the other hand the
compactness assumption is proscribed. All that will be clarified later.

One of our main objects is to investigate the asymptotic properties of $%
\mathcal{H=}\left( H_{\varepsilon }\right) _{\varepsilon \in E}$ for small
(or large) $\varepsilon $. So we are led to consider functions of the form $%
u\circ H_{\varepsilon }$, that is, such functions as $x\rightarrow u\left(
H_{\varepsilon }\left( x\right) \right) $ of $X$ into $\mathbb{C}$ (the
complex numbers), where $u$ is a real or complex function on $X$. However,
prior to this, one needs to provide $X$ with a positive Radon measure such
that if $u$ lies in $L^{p}\left( X\right) $ (resp. $L_{loc}^{p}\left(
X\right) $), $1\leq p\leq \infty $, then so also does $u\circ H_{\varepsilon
}$ for each $\varepsilon \in E$. One measure one naturally has in mind is an 
$\mathcal{H}$-invariant positive measure on $X$, that is, a positive Radon
measure $\lambda $ on $X$ such that $H_{\varepsilon }\left( \lambda \right)
=\lambda $ for all $\varepsilon \in E$ (see Remark \ref{rem1.1} below).
Unfortunately, as will be seen later, the only nonzero $\mathcal{H}$%
-invariant positive Radon measure on $X$ is the Dirac measure $\delta
_{\omega }$, where $\omega $ is the center of the absorptive action $%
\mathcal{H}$. In default of $\mathcal{H}$-invariant positive Radon measures
we consider the larger family of positive Radon measures $\lambda $ on $X$
satisfying:%
\begin{equation*}
\text{(H)\quad To each }\varepsilon \mathcal{\in }E\text{ corresponds some
constant }c\left( \varepsilon \right) >0\text{ such that }
\end{equation*}%
\begin{equation*}
H_{\varepsilon }\left( \lambda \right) \mathcal{=}c\left( \varepsilon
\right) \lambda \text{.}
\end{equation*}

\begin{remark}
\emph{\label{rem1.1} It is useful, perhaps, to recall that }$H_{\varepsilon
}\left( \lambda \right) $\emph{\ denotes the Radon measure on }$X$\emph{\
defined by }$H_{\varepsilon }\left( \lambda \right) \left( \varphi \right)
=\int \varphi \left( H_{\varepsilon }\left( x\right) \right) d\lambda \left(
x\right) $\emph{\ for }$\varphi \in \mathcal{K}\left( X\right) $\emph{\
(space of compactly supported continuous complex functions on }$X$\emph{).
In order that a complex function }$u$\emph{\ on }$X$\emph{\ be integrable
for }$H_{\varepsilon }\left( \lambda \right) $\emph{, it is necessary and
sufficient that the function }$u\circ H_{\varepsilon }$\emph{\ be integrable
for }$\lambda $\emph{, in which case }$\int u\left( H_{\varepsilon }\left(
x\right) \right) d\lambda \left( x\right) =\int u\left( x\right)
dH_{\varepsilon }\left( \lambda \right) \left( x\right) $\emph{.}
\end{remark}

A measure having property (H) is called \textit{homogeneous }for $\mathcal{H}
$, or $\mathcal{H}$-\textit{homogeneous}, or simply \textit{homogeneous}
when there is no danger of confusion. One major result achieved here is that
there is always some nontrivial homogeneous positive Radon measure on $X$
provided each point in $X$ has a countable base of neighbourhoods. It is not
out of interest to mention in passing that homogeneous measures are
particular cases of so-called quasi-invariant measures (see \cite{bib8},
p.23, and \cite{bib25}, p.237).

The concept of mean value is often crucial in most of the situations dealing
with asymptotic studies. It is shown here that $\mathcal{H}$ generates a
mean value on $X$ which extends various classical mean values available in
mathematical analysis. Finally, other important notions attached to
absorptive continuous $\mathbb{R}$-group actions on $X$, and which are
discussed here, are the notion of a homogenization algebra and that of
sigma-convergence introduced earlier in the context of numerical spaces (see 
\cite{bib18}). We prove a compactness result generalizing the two-scale
compactness theorem considered as the corner stone of a well-known
homogenization approach (see \cite[17, 19]{bib1}).

The rest of the study is organized as follows. Section 2 is concerned with
setting basic definitions and presenting preliminary results. Also, various
notions such as $\mathbb{R}$-groups and absorptive actions are illustrated
by concrete examples. In Section 3 we introduce homogeneous measures and we
establish, among other things, the existence of such measures. Section 4 is
devoted to the study of the mean value on $\left( X,\mathcal{H},\lambda
\right) $, where $X$ is a suitable locally compact space, $\mathcal{H}$ an
absorptive continuous $\mathbb{R}$-group action on $X$, and $\lambda $ a
nontrivial $\mathcal{H}$-homogeneous positive measure on $X$. Section 5 is
concerned with the study of absorptive continuous $\mathbb{R}$-group actions
on locally compact abelian groups. Finally, in Section 6 we discuss the
notion of a homogenization algebra and the subsequent concept of
sigma-convergence. We prove a compactness result generalizing the two-scale
compactness theorem and hence pointing out a close connection between
homogenization theory and absorptive continuous $\mathbb{R}$-group actions.

Except where otherwise stated, vector spaces throughout are considered over
the complex field, $\mathbb{C}$, and scalar functions assume complex values.
We will follow a standard notation. For convenience we will merely write $%
L^{p}\left( X\right) $ (resp. $L_{loc}^{p}\left( X\right) $) in place of $%
L^{p}\left( X;\mathbb{C}\right) $ (resp. $L_{loc}^{p}\left( X;\mathbb{C}%
\right) $), $1\leq p\leq +\infty $, where $X$ is a locally compact space
equipped with some\ positive Radon measure. We refer to \cite[5, 10]{bib4}
for integration theory.

\section{DEFINITIONS. BASIC RESULTS. EXAMPLES}

\subsection{\textbf{Absorptive continuous }$\mathbb{R}$\textbf{-group actions%
}. \textbf{Basic results}}

Throughout \ this subsection, $X$ denotes a locally compact space not
reduced to one point.

We start with the definition of an $\mathbb{R}$-group.

\begin{definition}
\emph{\label{def2.1} By an} $\mathbb{R}$-\textit{group }\emph{is meant any
abelian group }$E$ \emph{made up of real numbers, containing all positive
integers and satisfying the following conditions, where the group operation
is denoted multiplicatively (but the dot is omitted):}

\emph{(RG)}$_{1}$\emph{\ Provided with the natural order on} $\mathbb{R}$, $%
E $ \emph{is a totally ordered group.}

\emph{(RG)}$_{2}$\emph{\ Endowed with the relative usual topology on} $%
\mathbb{R}$, $E$ \emph{is a locally compact group.}

\emph{(RG)}$_{3}$\emph{\ There exists at least one nonconstant continuous
group} \emph{homomorphism }$h:E\rightarrow \mathbb{R}_{+}^{\ast }$ \emph{%
with the property that for each }$\alpha \mathfrak{\in }E$\emph{, the set} $%
E_{\alpha }\mathfrak{=}\left\{ \varepsilon \in E:\varepsilon \geq \alpha
\right\} $ \emph{is integrable for the measure }$h.m$\emph{, where }$m$\emph{%
\ denotes Haar measure on }$E$\emph{.}
\end{definition}

\textbf{Notation}. Given an $\mathbb{R}$-group $E$, we will denote by $e$
the identity of $E$, and by $\theta $ the greatest lower bound of $E$ in $%
\overline{\mathbb{R}}=\mathbb{R\cup }\left\{ -\infty ,+\infty \right\} $.
The inverse of $\varepsilon \in E$ will be denoted by $\varepsilon ^{-1}$,
i.e., $\varepsilon ^{-1}\varepsilon =e$.\bigskip

We now present a few typical $\mathbb{R}$-groups.

\begin{example}
\emph{\label{ex2.1} (The additive group} $\mathbb{R}$\emph{). The additive
group }$\mathbb{R}$ \emph{is the fundamental }$\mathbb{R}$\emph{-group.
Conditions (RG)}$_{1}$\emph{\ and (RG)}$_{2}$ \emph{are evident; condition
(RG)}$_{3}$\emph{\ is satisfied with }$h\left( \varepsilon \right) =\exp
\left( -r\varepsilon \right) $ $\left( \varepsilon \in \mathbb{R}\right) $, 
\emph{where} $r\in \mathbb{R}$, $r>0$ \emph{(note that Haar measure on} $%
\mathbb{R}$ \emph{is just Lebesgue measure). We have here }$e=0$ \emph{and} $%
\theta =-\infty $.
\end{example}

\begin{example}
\emph{\label{ex2.2} (The multiplicative group $\mathbb{R}_{+}^{\ast }$). The
multiplicative group }$\mathbb{R}_{+}^{\ast }$ \emph{of positive real
numbers is an} $\mathbb{R}$\emph{-group. We only need to check (RG)}$_{3}$, 
\emph{the remaining conditions of Definition \ref{def2.1} being obvious.
Recalling that Haar measure on} $\mathbb{R}_{+}^{\ast }$ \emph{is given by} $%
m\left( \varphi \right) =\int_{0}^{+\infty }\frac{\varphi \left( \varepsilon
\right) }{\varepsilon }d\varepsilon $, $\varphi \in \mathcal{K}\left( 
\mathbb{R}_{+}^{\ast }\right) $ \emph{(space of compactly supported
continuous complex functions on} $\mathbb{R}_{+}^{\ast }$\emph{), we see
that (RG)}$_{3}$ \emph{follows with }$h\left( \varepsilon \right) =\frac{1}{%
\varepsilon ^{r}}$\emph{\ }$\left( \varepsilon >0\right) $\emph{, where }$%
r\in \mathbb{R}$, $r>0$\emph{. Here }$e=1$\emph{\ and }$\theta =0$\emph{.}
\end{example}

\begin{example}
\emph{\label{ex2.3} (The discrete additive\ group $\mathbb{Z}$). The
discrete additive group }$\mathbb{Z}$ \emph{is an} $\mathbb{R}$\emph{-group.
Indeed, (RG)}$_{1}$\emph{-(RG)}$_{2}$ \emph{are evident and (RG)}$_{3}$\emph{%
\ follows by taking} $h\left( n\right) =a^{n}$ $\left( n\in \mathbb{Z}%
\right) $ \emph{with }$0<a<1$\emph{\ and recalling that Haar measure on} $%
\mathbb{Z}$\emph{\ is given by} $m\left( \varphi \right) =\sum_{n\in \mathbb{%
Z}}\varphi \left( n\right) $ $\left( \varphi \in \mathcal{K}\left( \mathbb{Z}%
\right) \right) $. \emph{Here }$e=0$\emph{\ and }$\theta =-\infty $\emph{.}
\end{example}

At the present time, let $E$ be an $\mathbb{R}$-group, and let $\mathcal{H=}%
\left( H_{\varepsilon }\right) _{\varepsilon \in E}$ be an action of $E$ on $%
X$.

\begin{definition}
\label{def2.2} \emph{Let }$T$\emph{\ and }$S$\emph{\ be two subsets of }$X$%
\emph{. One says that:}

\emph{(i) }$T$\emph{\ }\textit{absorbs}\emph{\ }$S$\emph{\ if there is some }%
$\alpha \in E$\emph{\ such that }$H_{\varepsilon ^{-1}}\left( S\right)
\subset T$\emph{\ for }$\varepsilon \leq \alpha $\emph{;}

\emph{(ii) }$T$\emph{\ is }\textit{absorbent}\emph{\ if }$T$\emph{\ absorbs
any singleton }$\left\{ x\right\} $\emph{\ }$\left( x\in X\right) $\emph{;}

\emph{(iii) }$T$\emph{\ is }\textit{balanced}\emph{\ if }$H_{\varepsilon
^{-1}}\left( T\right) \subset T$\emph{\ for }$\varepsilon \leq e$\emph{.}
\end{definition}

\begin{remark}
\emph{\label{rem2.1} We note in passing that this terminology is justifiably
borrowed from the theory of Topological Vector Spaces. On the other hand,
one should have said }$\mathcal{H}$-\textit{absorbs, }$\mathcal{H}$-\textit{%
absorbent, }$\mathcal{H}$-\textit{balanced, }\emph{in order to emphasize
that these notions are intrinsically attached to }$\mathcal{H}$\emph{.
However, we will omit }$\mathcal{H}$\emph{\ as long as there is no danger of
confusion.}
\end{remark}

The verification of the following proposition is a routine exercise left to
the reader.

\begin{proposition}
\label{pr2.1} The following hold true:

(i) The set $T\subset X$ is balanced if and only if $\varepsilon \leq
\varepsilon ^{\prime }$ implies $H_{\varepsilon ^{\prime }}\left( T\right)
\subset H_{\varepsilon }\left( T\right) $.

(ii) If $T$ is balanced and if $H_{\alpha ^{-1}}\left( S\right) \subset T$
for some $\alpha \in E$, where $S\subset X$ is given, then $T$ absorbs $S$.

(iii) Any union (resp. intersection)\ of balanced sets is balanced.

(iv) Any finite intersection of absorbent sets is absorbent.
\end{proposition}

We turn now to the concept of absorptiveness.

\begin{definition}
\emph{\label{def2.3}} \emph{The action} $\mathcal{H}$ \emph{is termed }%
\textit{absorptive}\emph{\ if there exists some }$\omega \in X$\emph{\ with
the following property:}

\emph{(ABS)\quad For any neighbourhood }$V$\emph{\ of }$\omega $\emph{\ and
for any }$x\in X$\emph{, there are some neighbourhood }$U$\emph{\ of }$x$%
\emph{\ and some }$\alpha \in E$\emph{\ such that }$H_{\varepsilon
^{-1}}\left( U\right) \subset V$\emph{\ for }$\varepsilon \leq \alpha $\emph{%
.}
\end{definition}

\begin{remark}
\emph{\label{rem2.2}} \emph{Suppose a point }$\omega $\emph{\ in }$X$\emph{\
satisfies (ABS). Then, for each }$x\in X$\emph{, we have }$\lim_{\varepsilon
\rightarrow \theta }H_{\varepsilon ^{-1}}\left( x\right) =\omega $\emph{\
(i.e., as }$\varepsilon \rightarrow \theta $\emph{, }$H_{\varepsilon
^{-1}}\left( x\right) $\emph{\ tends to }$\omega $\emph{\ in }$X$\emph{).}
\end{remark}

Before we proceed any further let us prove one simple but basic result.

\begin{proposition}
\label{pr2.2} Suppose the action $\mathcal{H}$ is continuous and absorptive.
Let $\omega $ satisfy \emph{(ABS). }Then:

(i) $\omega $ is invariant for $\mathcal{H}$, i.e., $H_{\varepsilon }\left(
\omega \right) =\omega $ for all $\varepsilon \in E$;

(ii) $\omega $ is unique.
\end{proposition}

\begin{proof}
For any arbitrarily fixed $r\mathfrak{\in }E$, we have $\lim_{\varepsilon
\rightarrow \theta }H_{r}\left( H_{\varepsilon ^{-1}}\left( x\right) \right)
=H_{r}\left( \omega \right) $, as is straightforward by Remark \ref{rem2.2}
and use of the fact that $H_{r}$ maps continuously $X$ into itself. But the
left-hand side of the preceding equality is precisely $\lim_{\varepsilon
\rightarrow \theta }H_{\left( \varepsilon r^{-1}\right) ^{-1}}\left(
x\right) =\omega $, thanks to Remark \ref{rem2.2}, once again. Hence (i)
follows. Finally, (ii) is immediate by combining (i) with Remark \ref{rem2.2}%
.
\end{proof}

\begin{definition}
\emph{\label{def2.4} Suppose the action }$\mathcal{H}$ \emph{is continuous
and absorptive.}

\emph{(i) The point }$\omega $\emph{\ satisfying (ABS) is called the }%
\textit{center}\emph{\ of }$\mathcal{H}$.

\emph{(ii)} \emph{A set }$T$\emph{\ is called }\textit{bounded}\emph{\ (for }%
$\mathcal{H}$\emph{) if }$T$\emph{\ is absorbed by any neighbourhood of }$%
\omega $\emph{.}

\emph{(iii) A set }$T$\emph{\ is called }\textit{elementary}\emph{\ (for} $%
\mathcal{H}$\emph{) if }$T$\emph{\ is a balanced relatively compact
neighbourhood of }$\omega $\emph{.}
\end{definition}

We are now ready to develop a number of basic results.

\begin{proposition}
\label{pr2.3} Suppose the action $\mathcal{H}$ is continuous and absorptive,
and let $\omega $ be its center. The following assertions are true:

(i) There is a neighbourhood base at $\omega $ made up of balanced absorbent
sets.

(ii) A set $T\subset X$ is bounded (for $\mathcal{H}$) if and only if it is
relatively compact.

(iii) Elementary sets (for $\mathcal{H}$) do exist. Furthermore, if $F$ is
one such set, then, as $n$ ranges over the positive integers, the sets $%
H_{n}\left( F\right) $ form a neighbourhood base at $\omega $ whereas the
sets $H_{n^{-1}}\left( F\right) $ form a covering of $X$.
\end{proposition}

\begin{proof}
(i): Let $V$ be any arbitrary neighbourhood of $\omega $. According to the
absorptiveness, choose a\ neighbourhood $U_{0}$ of $\omega $ and some $%
\alpha \in E$ such that $H_{\varepsilon ^{-1}}\left( U_{0}\right) \subset V$
for $\varepsilon \leq \alpha $. Clearly the set $U_{1}=H_{\alpha
^{-1}}\left( U_{0}\right) $ is a neighbourhood of $\omega $ and further $%
H_{\varepsilon ^{-1}}\left( U_{1}\right) \subset V$ for $\varepsilon \leq e$%
. Hence the set $U={\LARGE \cup }_{\varepsilon \leq e}H_{\varepsilon
^{-1}}\left( U_{1}\right) $ is a balanced neighbourhood of $\omega $ and $%
U\subset V$. Seing that any neighbourhood of $\omega $ is absorbent (this is
immediate by (ABS)), (i) follows. (ii): Suppose $T$ is bounded. Then $%
T\subset H_{\varepsilon }\left( V\right) $ for some $\varepsilon \in E$,
where $V$ is a compact neighbourhood of $\omega $. This shows that $T$ is
relatively compact. Conversely\ suppose\ $T$ is relatively compact. We may
assume without loss of generality that $T$ is compact. On the other hand,
let $U$ be a balanced neighbourhood of $\omega $, which we may assume to be
open (see the\ proof of (i) above). Since $U$ is absorbent, to each $x\in T$
there is assigned some $\alpha _{x}\in E$ such that $x\in H_{\alpha
_{x}}\left( U\right) $. By compactness this yields a finite family $\left\{
\alpha _{i}\right\} _{1\leq i\leq n}$ in $E$ such that the sets $H_{\alpha
_{i}}\left( U\right) $ $\left( 1\leq i\leq n\right) $ form a covering of $T$%
. It follows that $T\subset H_{\alpha }\left( U\right) $ with $\alpha
=\min_{1\leq i\leq n}\alpha _{i}$ (use part (i) of Proposition \ref{pr2.1}).
Hence $U$ absorbs $T$ (use part (ii) of Proposition \ref{pr2.1}). Therefore,
(ii) follows by the fact that each neighbourhood of $\omega $ contains one
such $U$, as was established earlier.

\noindent (iii): If $K$ is a compact neighbourhood of $\omega $, then, as
shown in (i), there is some $F\subset K$ which is a balanced neighbourhood
of $\omega $. $F$ is clearly an elementary set. Now, let $V$ be any
arbitrarily given neighbourhood of $\omega $. In view of the boundedness of $%
F$ (see (ii)), there is some $\alpha \in E$ such that $H_{\varepsilon
^{-1}}\left( F\right) \subset V$ for $\varepsilon \leq \alpha $. Hence $%
H_{n}\left( F\right) \subset V$ for any positive integer $n\geq \alpha ^{-1}$%
. Finally, recalling that $F$ is absorbent (as a neighbourhood of $\omega $)
we see that if $x\in X$ is freely fixed, then $x\in H_{\varepsilon }\left(
F\right) $ for $\varepsilon \leq r$, where $r$ is a suitable element of $E$.
Hence $x\in H_{n^{-1}}\left( F\right) $ for any positive integer $n\geq
r^{-1}$. This completes the proof of (iii).
\end{proof}

This proposition has a few important corollaries. So, in what follows, the
notation and hypotheses are as in Proposition \ref{pr2.3}.

\begin{corollary}
\label{co2.1} $X$ is noncompact, nondiscrete, and $\sigma $-compact (i.e., $%
X $ is expressible as the union of a countable family of compact subspaces).
\end{corollary}

\begin{proof}
We begin by verifying that $X$ is noncompact. Let us assume the contrary.
Let $\mathcal{B}$ denote the set of all closed neighbourhoods of $\omega $.
In view of part (ii) of Proposition \ref{pr2.3}, $X$ is absorbed by each
member of $\mathcal{B}$. We deduce that $X\subset {\LARGE \cap }_{K\in 
\mathcal{B}}K=\left\{ \omega \right\} $ ($X$ being Hausdorff). This
contradicts the fact that $X$ is not reduced to one point, and so $X$ is
noncompact. We next show that $X$ is nondiscrete. This is straighforward.
Indeed, if we assume that $X$ is discrete, then the set $\left\{ \omega
\right\} $ is a neighbourhood of $\omega $ and therefore $\left\{ \omega
\right\} $ is absorbent. Combining this with part (i) of Proposition \ref%
{pr2.2}, it follows that $X$ is reduced to $\left\{ \omega \right\} $, a
contradiction. We finally check that $X$ is $\sigma $-compact. Let $F$ be an
elementary set. It is obvious that $\overline{F}$ (the closure of $F$) is an
elementary set, which moreover is compact. It follows that for each positive
integer $n$, the set $H_{n^{-1}}\left( \overline{F}\right) $ is compact. But
then, the corresponding family is a covering of $X$, according to part (iii)
of Proposition \ref{pr2.3}. This shows that $X$ is $\sigma $-compact.
\end{proof}

\begin{corollary}
\label{co2.2} For each $x\in X$, $x\neq \omega $, we have $\lim_{\varepsilon
\rightarrow \theta }H_{\varepsilon }\left( x\right) =\infty $, where $\infty 
$ denotes the point at infinity of the Alexandroff compactification of $X$.
\end{corollary}

\begin{proof}
Let $x\in X$ with $x\neq \omega $. Let $K$ be any arbitrary compact set in $%
X $. Let $V$ be a neighbourhood of $\omega $ not containing $x$. Thanks to
the boundedness of $K$ (part (ii) of Proposition \ref{pr2.3}), there is some 
$\alpha \in E$ such that $H_{\varepsilon }\left( X\backslash V\right)
\subset X\backslash K$ for $\varepsilon \leq \alpha $. Hence $H_{\varepsilon
}\left( x\right) \in X\backslash K$ for $\varepsilon \leq \alpha $. This
shows the corollary.
\end{proof}

\begin{corollary}
\label{co2.3} Let $F$ be an open elementary set. To each compact set $%
K\subset X$ there is assigned some integer $n\geq 1$ such that $K\subset
H_{n^{-1}}\left( F\right) $.
\end{corollary}

\begin{proof}
The family $\left\{ H_{n^{-1}}\left( F\right) \right\} $, where $n$ ranges
over the positive integers, is on one hand a covering of $X$ (part (iii) of
Proposition \ref{pr2.3}), on the other hand an increasing sequence (for
inclusion). Hence the corollary follows.
\end{proof}

\begin{corollary}
\label{co2.4} There is a countable neighbourhood base at $\omega $
consisting of compact (resp. open) elementary sets.
\end{corollary}

\begin{proof}
Let $F$ be an elementary set. Observing that for each $\varepsilon \in E$
the compact set $H_{\varepsilon }\left( \overline{F}\right) $ and the open
set $H_{\varepsilon }(\mathring{F})$ ($\mathring{F}$ the interior of $F$)
are elementary sets, we are led to the desired result by [part (iii) of]
Proposition \ref{pr2.3}.
\end{proof}

By way of illustration, let us exhibit one typical absorptive continuous $%
\mathbb{R}$-group action.

\begin{example}
\label{ex2.4} \emph{On the }$N$\emph{-dimensional numerical space} $\mathbb{R%
}^{N}$ $(N\geq 1)$ \emph{we consider the action} $\mathcal{H=}\left(
H_{\varepsilon }\right) _{\varepsilon >0}$ of $E=\mathbb{R}_{+}^{\ast }$ 
\emph{given by} 
\begin{equation*}
H_{\varepsilon }\left( x\right) =\left( \frac{x_{1}}{\varepsilon ^{r_{1}}}%
,...,\frac{x_{N}}{\varepsilon ^{r_{N}}}\right) \text{, }x=\left(
x_{1},...,x_{N}\right) \in \mathbb{R}^{N}\text{, }\left( \varepsilon
>0\right) \text{,}
\end{equation*}%
\emph{where} $r_{i}\in \mathbb{N}$, $r_{i}>0$ $\left( i=1,...,N\right) $. 
\emph{The action} $\mathcal{H}$ \emph{is continuous and absorptive, with
center} \emph{the origin} $\omega =\left( 0,...,0\right) $ \emph{in} $%
\mathbb{R}^{N}$. \emph{The most typical case is when }$r_{i}=1$ \emph{(}$%
i=1,...,N$\emph{), i.e.,} $H_{\varepsilon }\left( x\right) =\frac{x}{%
\varepsilon }$, $x\in \mathbb{R}^{N}$.
\end{example}

Further examples will be presented in the next subsection. Now, we will end
the present subsection by turning our attention to proving that the product
of a finite family of absorptive continuous actions of the same $\mathbb{R}$%
-group is still continuous and absorptive. More precisely, let $\left\{
X_{i}\right\} _{1\leq i\leq n}$ be a finite family of locally compact
spaces. For each integer $i$ ($1\leq i\leq n$), let $\mathcal{H}_{i}=\left(
H_{\varepsilon }^{i}\right) _{\varepsilon \in E}$ be an action on $X_{i}$ of
a fixed $\mathbb{R}$-group $E$. Put $X=X_{1}\times {\small \cdot \cdot \cdot 
}\times X_{n}$ (a locally compact space with the product topology) and $%
H_{\varepsilon }=H_{\varepsilon }^{1}\times {\small \cdot \cdot \cdot }%
\times H_{\varepsilon }^{n}$ (direct product) for fixed $\varepsilon \in E$.
It is worth recalling that $H_{\varepsilon }$ is the mapping of $X$ into
itself defined by 
\begin{equation*}
H_{\varepsilon }\left( x\right) =\left( H_{\varepsilon }^{1}\left(
pr_{1}\left( x\right) \right) ,...,H_{\varepsilon }^{n}\left( pr_{n}\left(
x\right) \right) \right) \qquad \left( x\in X\right)
\end{equation*}%
where $pr_{i}$ denotes the natural projection of $X$ onto $X_{i}$. The
family $\mathcal{H=}\left( H_{\varepsilon }\right) _{\varepsilon \in E}$ is
an action of $E$ on $X$, and is referred to as the product of the actions $%
\mathcal{H}_{i}$ ($1\leq i\leq n$), sometimes denoted by $\mathcal{H=\sqcap }%
_{i=1}^{n}\mathcal{H}_{i}=\mathcal{H}_{1}\times {\small \cdot \cdot \cdot }%
\times \mathcal{H}_{n}$.

\begin{proposition}
\label{pr2.4} Let the notation and hypotheses be as above. Assume moreover
that each action $\mathcal{H}_{i}$ is continuous and absorptive, and of
center $\omega _{i}$. Then, the product action $\mathcal{H}$ is continuous
and absorptive, and of center $\omega =\left( \omega _{1},...,\omega
_{n}\right) $.
\end{proposition}

\begin{proof}
We begin by showing the continuity. To this end we introduce the mappings $%
H:E\times X\rightarrow X$, $H^{i}:E\times X_{i}\rightarrow X_{i}$ ($1\leq
i\leq n$) and $f_{i}:E\times X\rightarrow E\times X_{i}$ ($1\leq i\leq n$)
defined respectively by%
\begin{equation*}
H\left( \varepsilon ,x\right) =H_{\varepsilon }\left( x\right) \text{, }%
H^{i}\left( \varepsilon ,t\right) =H_{\varepsilon }^{i}\left( t\right) \text{
and }f_{i}\left( \varepsilon ,x\right) =\left( \varepsilon ,pr_{i}\left(
x\right) \right)
\end{equation*}%
for $\varepsilon \in E$, $x\in X$ and $t\in X_{i}$. Clearly%
\begin{equation*}
H\left( \varepsilon ,x\right) =\left( H^{1}\left( f_{1}\left( \varepsilon
,x\right) \right) ,...,H^{n}\left( f_{n}\left( \varepsilon ,x\right) \right)
\right)
\end{equation*}%
for $\left( \varepsilon ,x\right) \in E\times X$, and the claimed continuity
follows at once.

We now check the absorptiveness. Let $V$ be a neighbourhood of $\omega
=\left( \omega _{1},...,\omega _{n}\right) $ in $X$. Let $V_{1}\times 
{\small \cdot \cdot \cdot }\times V_{n}\subset V$, where $V_{i}$ is some
neighbourhood of $\omega _{i}$ in $X_{i}$ ($1\leq i\leq n$). Finally, let $%
x\in X$. For fixed $i$, the absorptiveness of $\mathcal{H}_{i}$ yields some $%
\alpha _{i}\mathfrak{\in }E$ and some neighbourhood $U_{i}$ of $pr_{i}\left(
x\right) $ in $X_{i}$ such that $U_{i}\subset H_{\varepsilon }^{i}\left(
V_{i}\right) $ for $\varepsilon \leq \alpha _{i}$. Hence $U\subset
H_{\varepsilon }\left( V\right) $ for $\varepsilon \leq \alpha $, where $%
\alpha =\min_{1\leq i\leq n}\alpha _{i}$ and $U=U_{1}\times {\small \cdot
\cdot \cdot }\times U_{n}$ (a neighbourhood of $x$ in $X$). This shows the
absorptiveness of $\mathcal{H}$.
\end{proof}

\subsection{\textbf{Contraction flows}}

The study of contraction flows provides further examples of absorptive
continuous $\mathbb{R}$-group actions.

\begin{definition}
\label{def2.5} \emph{By a contraction flow we will mean a pair} $\left( X,%
\mathcal{H}\right) $ \emph{in which:}

\noindent \emph{(CF)}$_{1}$\emph{\ }$X$\emph{\ is a complete locally compact
metric space not reduced\ to one\ point;}

\noindent \emph{(CF)}$_{2}$ $\mathcal{H=}\left( H_{\varepsilon }\right)
_{\varepsilon \in E}$ \emph{is an action of an} $\mathbb{R}$-\emph{group, }$%
E $\emph{, on }$X$\emph{\ with the following properties:}

\emph{(i) For each fixed }$x\in X$\emph{, the map }$\varepsilon \rightarrow
H_{\varepsilon }\left( x\right) $\emph{\ sends continuously }$E$\emph{\ into 
}$X$\emph{.}

\emph{(ii) For each }$\varepsilon \in E$\emph{, we have}%
\begin{equation*}
\sup_{x,y\in X,\text{ }x\neq y}\frac{d\left( H_{\varepsilon }\left( x\right)
,H_{\varepsilon }\left( y\right) \right) }{d\left( x,y\right) }<+\infty 
\text{,}
\end{equation*}%
\emph{where }$d$\emph{\ is the metric on }$X$\emph{.}

\emph{(iii) We have}%
\begin{equation*}
\lim_{\varepsilon \rightarrow +\infty }\sup_{x,y\in X,\text{ }x\neq y}\frac{%
d\left( H_{\varepsilon }\left( x\right) ,H_{\varepsilon }\left( y\right)
\right) }{d\left( x,y\right) }=0\text{.}
\end{equation*}%
\bigskip
\end{definition}

In what follows, $\left( X,\mathcal{H}\right) $ denotes a contraction flow,
and the basic notation is as in Definition \ref{def2.5}. For each $%
\varepsilon \in E$, let 
\begin{equation*}
l\left( \varepsilon \right) =\sup_{x,y\in X,\text{ }x\neq y}\frac{d\left(
H_{\varepsilon }\left( x\right) ,H_{\varepsilon }\left( y\right) \right) }{%
d\left( x,y\right) }\text{,}
\end{equation*}%
which defines a nonnegative real function on $E$ denoted $l$.

\begin{remark}
\emph{\label{rem2.3} For fixed }$\varepsilon \in E$\emph{, we have }$d\left(
H_{\varepsilon }\left( x\right) ,H_{\varepsilon }\left( y\right) \right)
\leq l\left( \varepsilon \right) d\left( x,y\right) $\emph{\ for all }$%
x,y\in X$\emph{. Moreover, }$l\left( \varepsilon \right) $\emph{\ is the
smallest constant for which this inequality holds true.}
\end{remark}

A few useful properties of the function $l$ are collected below.

\begin{lemma}
\label{lem2.1} The following assertions hold true.

(i) The function $l$ actually maps $E$ into $\mathbb{R}_{+}^{\ast }$.

(ii) We have:%
\begin{equation}
l\left( \varepsilon \varepsilon ^{\prime }\right) \leq l\left( \varepsilon
\right) l\left( \varepsilon ^{\prime }\right) \text{ for all }\varepsilon 
\text{, }\varepsilon ^{\prime }\in E\text{,}  \label{eq2.1}
\end{equation}%
\begin{equation}
l\left( e\right) =1\text{,\quad }  \label{eq2.2}
\end{equation}%
\begin{equation}
\lim_{\varepsilon \rightarrow \theta }l\left( \varepsilon ^{-1}\right) =0%
\text{, }  \label{eq2.3}
\end{equation}%
\begin{equation}
\sup_{\varepsilon \leq \alpha }l\left( \varepsilon ^{-1}\right) <+\infty 
\text{\qquad for any }\alpha \in E\text{.}  \label{eq2.4}
\end{equation}
\end{lemma}

\begin{proof}
In fact, (i) amounts to saying that $l\left( \varepsilon \right) >0$ for all 
$\varepsilon \in E$. But this follows immediately by the fact that each $%
H_{\varepsilon }$ is one to one. Let us turn now to the verification of
(ii): (\ref{eq2.1})-(\ref{eq2.2}) follow immediately by Remark \ref{rem2.3},
and (\ref{eq2.3}) is a direct consequence of property (iii) in Definition %
\ref{def2.5}. It remains to verify (\ref{eq2.4}). Based on (\ref{eq2.3}),
let us fix some $\varepsilon _{0}\in E$ such that $l\left( s^{-1}\right)
\leq 1$ for $s\leq \varepsilon _{0}$. With this in mind, let now $\alpha \in
E$. Consider any arbitrary $\varepsilon \in E$ with $\varepsilon \leq \alpha 
$. Let $t=\alpha ^{-1}\varepsilon _{0}\varepsilon $. Noting that $l\left(
t^{-1}\right) \leq 1$, and recalling (\ref{eq2.1}), we see at once that $%
l\left( \varepsilon ^{-1}\right) =l\left( \varepsilon _{0}\alpha
^{-1}t^{-1}\right) \leq l\left( \varepsilon _{0}\alpha ^{-1}\right) l\left(
t^{-1}\right) \leq l\left( \varepsilon _{0}\alpha ^{-1}\right) $. Hence (\ref%
{eq2.4}) follows.
\end{proof}

We are now in a position to prove the following result.

\begin{proposition}
\label{pr2.5} The action $\mathcal{H}$ is continuous and absorptive.
\end{proposition}

\begin{proof}
We begin by proving the continuity. Let $\left( \varepsilon
_{0},x_{0}\right) $ be freely fixed in $E\times X$. Choose $\alpha \in E$
with $\varepsilon _{0}<\alpha $, and put $J=(\theta ,\alpha ]\cap E$. The
set $J$ is a neighbourhood of $\varepsilon _{0}$ in $E$ and, thanks to (\ref%
{eq2.4}), we have $d\left( H_{\varepsilon }\left( x\right) ,H_{\varepsilon
}\left( y\right) \right) \leq c_{0}d\left( x,y\right) $ for $\varepsilon \in
J$ and $x$, $y\in X$, where the constant $c_{0}$ depends only on $%
\varepsilon _{0}$ and $\alpha $. Hence, given $\left( \varepsilon ,x\right)
\in J\times X$, by the inequality%
\begin{equation*}
d\left( H_{\varepsilon }\left( x\right) ,H_{\varepsilon _{0}}\left(
x_{0}\right) \right) \leq d\left( H_{\varepsilon }\left( x\right)
,H_{\varepsilon }\left( x_{0}\right) \right) +d\left( H_{\varepsilon }\left(
x_{0}\right) ,H_{\varepsilon _{0}}\left( x_{0}\right) \right)
\end{equation*}%
we have%
\begin{equation*}
d\left( H_{\varepsilon }\left( x\right) ,H_{\varepsilon _{0}}\left(
x_{0}\right) \right) \leq c_{0}d\left( x,x_{0}\right) +d\left(
H_{\varepsilon }\left( x_{0}\right) ,H_{\varepsilon _{0}}\left( x_{0}\right)
\right) \text{.}
\end{equation*}%
From which we deduce, using part (i) of Definition \ref{def2.5}, that%
\begin{equation*}
d\left( H_{\varepsilon }\left( x\right) ,H_{\varepsilon _{0}}\left(
x_{0}\right) \right) \rightarrow 0\text{ as }\left( \varepsilon ,x\right)
\rightarrow \left( \varepsilon _{0},x_{0}\right) \text{ in }E\times X\text{.}
\end{equation*}%
This shows the continuity of the action $\mathcal{H}$. We next show the
absorptiveness. We will first establish the existence of an invariant point $%
\omega \in X$ for $\mathcal{H}$. For this purpose, let $0<k<1$. According to
(\ref{eq2.3}), there is some $\varepsilon _{0}\in E$ such that%
\begin{equation*}
d\left( H_{\varepsilon ^{-1}}\left( x\right) ,H_{\varepsilon ^{-1}}\left(
y\right) \right) \leq kd\left( x,y\right)
\end{equation*}%
for $\varepsilon \in I=(\theta ,\varepsilon _{0}]\cap E$ and $x,y\in X$.
Therefore, by the classical fixed-point theorem, there exists a map $%
\varepsilon \rightarrow a\left( \varepsilon \right) $ of $I$ into $X$ such
that $H_{\varepsilon ^{-1}}\left( a\left( \varepsilon \right) \right)
=a\left( \varepsilon \right) $ for $\varepsilon \in I$. Furthermore, the
same theorem says that for fixed $\varepsilon \in I$, the point $a\left(
\varepsilon \right) $ is the unique solution of the equation $H_{\varepsilon
^{-1}}\left( x\right) =x$\quad $\left( x\in X\right) $. This being so, fix
freely $\alpha $, $\beta \in I$. Clearly%
\begin{equation*}
H_{\alpha ^{-1}}\left( H_{\beta ^{-1}}\left( a\left( \alpha \right) \right)
\right) =H_{\beta ^{-1}}\left( a\left( \alpha \right) \right) \text{.}
\end{equation*}%
By unicity, as pointed out above, this implies $H_{\beta ^{-1}}\left(
a\left( \alpha \right) \right) =a\left( \alpha \right) $. But, $H_{\beta
^{-1}}\left( a\left( \beta \right) \right) =a\left( \beta \right) $. Hence $%
a\left( \alpha \right) =a\left( \beta \right) $, and that for all $\alpha $, 
$\beta \in I$. This means that the map $\varepsilon \rightarrow a\left(
\varepsilon \right) $ is constant. In other words, there is a unique point $%
\omega \in X$ such that $H_{\varepsilon ^{-1}}\left( \omega \right) =\omega $
for all $\varepsilon \in I$. It trivially follows that $H_{\varepsilon
}\left( \omega \right) =\omega $ for all $\varepsilon \in I$. But this
extends to the whole $E$ because $\varepsilon \in E\backslash I$ implies $%
\varepsilon ^{-1}\varepsilon _{0}^{2}\in I$ and therefore $H_{\varepsilon
}\left( \omega \right) =\omega $. This shows the existence of an invariant
point $\omega \in X$ and it is clear that the latter is unique.

Now, let $V$ be a neighbourhood of $\omega $. We may assume without loss of
generality that $V$ is the open ball with center $\omega $ and radius some $%
r>0$, i.e.,%
\begin{equation*}
V=B\left( \omega ,r\right) \equiv \left\{ y\in X:d\left( \omega ,y\right)
<r\right\} \text{.}
\end{equation*}%
Let $x\in X$. The aim is to find some neighbourhood $U$ of $x$ and some $%
\alpha \in E$ such that $H_{\varepsilon ^{-1}}\left( U\right) \subset V$ for 
$\varepsilon \leq \alpha $. Let us first assume that $x=\omega $. By the
obvious relation $H_{\varepsilon ^{-1}}\left( B\left( \omega ,r\right)
\right) \subset B\left( \omega ,rl\left( \varepsilon ^{-1}\right) \right) $
and use of (\ref{eq2.3}), one quickly arrives at $H_{\varepsilon
^{-1}}\left( V\right) \subset V$ for $\varepsilon \leq \alpha $, where $%
\alpha \in E$ is chosen in such a way that $l\left( \varepsilon ^{-1}\right)
<1$ for $\varepsilon \leq \alpha $. We assume now that $x\neq \omega $.
Using (\ref{eq2.3}), once again, choose some $\alpha \in E$ such that $%
l\left( \varepsilon ^{-1}\right) \leq \frac{r}{2d\left( \omega ,x\right) }$
for $\varepsilon \leq \alpha $, and bear then in mind that $d\left( \omega
,x\right) <\frac{r}{c}$, where $c=\sup_{\varepsilon \leq \alpha }l\left(
\varepsilon ^{-1}\right) $ (use (\ref{eq2.4})). Finally, choose some
constant $r_{1}>0$ such that $r_{1}+d\left( \omega ,x\right) <\frac{r}{c}$.
This being so, fix freely $\varepsilon \leq \alpha $. If $y\in B\left(
x,r_{1}\right) =\left\{ z\in X:d\left( x,z\right) <r_{1}\right\} $, then, by
using%
\begin{equation*}
d\left( \omega ,H_{\varepsilon ^{-1}}\left( y\right) \right) \leq d\left(
\omega ,H_{\varepsilon ^{-1}}\left( x\right) \right) +d\left( H_{\varepsilon
^{-1}}\left( x\right) ,H_{\varepsilon ^{-1}}\left( y\right) \right)
\end{equation*}%
we get 
\begin{equation*}
d\left( \omega ,H_{\varepsilon ^{-1}}\left( y\right) \right) \leq l\left(
\varepsilon ^{-1}\right) \left[ d\left( \omega ,x\right) +d\left( x,y\right) %
\right]
\end{equation*}%
\begin{equation*}
\qquad \quad \quad <l\left( \varepsilon ^{-1}\right) \left[ d\left( \omega
,x\right) +r_{1}\right]
\end{equation*}%
\begin{equation*}
<r\text{.}\qquad \qquad
\end{equation*}%
Hence $H_{\varepsilon ^{-1}}\left( B\left( x,r_{1}\right) \right) \subset V$
for $\varepsilon \leq \alpha $. This shows the absorptiveness.
\end{proof}

Thus, the notion of a contraction flow provides an example of an absorptive
continuous $\mathbb{R}$-group action. Let us illustrate this still more.

\begin{example}
\emph{\label{ex2.5} Let} $\mathcal{H=}\left( H_{\varepsilon }\right)
_{\varepsilon \in E}$ \emph{be an action of }$E$ \emph{on} $\mathbb{R}^{N}$, 
$E$ \emph{being a given} $\mathbb{R}$\emph{-group and} $\mathbb{R}^{N}$ 
\emph{(with} $N\geq 1$) \emph{being equipped with the Euclidean metric. We
suppose that:\medskip }

\emph{(1) Each} $H_{\varepsilon }$ \emph{is an automorphism of} $\mathbb{R}%
^{N}$ \emph{(viewed as an }$N$\emph{-dimensional vector space);\medskip }

\emph{(2) for each }$x\in \mathbb{R}^{N}$, \emph{the map }$\varepsilon
\rightarrow H_{\varepsilon }\left( x\right) $\emph{\ sends \ continuously} $%
E $ \emph{into} $\mathbb{R}^{N}$;\medskip

\emph{(3) }$\lim_{\varepsilon \rightarrow +\infty }\left\vert H_{\varepsilon
}\left( x\right) \right\vert =0$\emph{\ for each }$x\in \mathbb{R}^{N}$, 
\emph{where }$\left\vert {\small \cdot }\right\vert $\emph{\ denotes the
Euclidean} \emph{norm in} $\mathbb{R}^{N}$.

\emph{We want to show that this action is continuous and absorptive. Thanks\
to Proposition \ref{pr2.5}, it suffices to verify that }$\left( \mathbb{R}%
^{N},\mathcal{H}\right) $ \emph{is a contraction flow. Clearly we only need
to show that conditions (ii)-(iii) of (CF)}$_{2}$\emph{\ (Definition \ref%
{def2.5}) are fulfilled} \emph{when }$\mathcal{H}$ \emph{is the action under
consideration here and }$X$\emph{\ is} $\mathbb{R}^{N}$. \emph{First of all,}
$\mathbb{R}^{N}$ \emph{being provided with the canonical basis, the
automorphism }$H_{\varepsilon }$\emph{\ is represented by an }$N\times N$%
\emph{\ real matrix} $B\left( \varepsilon \right) =\left( b_{ij}\left(
\varepsilon \right) \right) _{1\leq i,j\leq N}$\emph{, so that}%
\begin{equation*}
H_{\varepsilon }^{i}\left( x\right) =\sum_{j=1}^{N}b_{ij}\left( \varepsilon
\right) x_{j}\text{ \emph{for }}x=\left( x_{1},...,x_{N}\right) \in \mathbb{R%
}^{N}\qquad \left( 1\leq i\leq N\right) \text{,}
\end{equation*}%
\emph{where }$H_{\varepsilon }^{i}$\emph{\ is the }$i$\emph{-th component of 
}$H_{\varepsilon }=\left( H_{\varepsilon }^{1},...,H_{\varepsilon
}^{N}\right) $\emph{. One deduces immediately that}%
\begin{equation}
\lim_{\varepsilon \rightarrow +\infty }b_{ij}\left( \varepsilon \right)
=0\qquad \left( 1\leq i,j\leq N\right) \text{.}  \label{eq2.5}
\end{equation}%
\emph{Actually, (\ref{eq2.5}) is equivalent to (3). On the other hand,} 
\begin{equation*}
H_{\varepsilon }^{i}\left( x\right) -H_{\varepsilon }^{i}\left( y\right)
=\sum_{j=1}^{N}b_{ij}\left( \varepsilon \right) \left( x_{j}-y_{j}\right) 
\text{\quad \emph{for} }x\text{, }y\in \mathbb{R}^{N}\text{, }\varepsilon
\in E\text{.}
\end{equation*}%
\emph{Hence}%
\begin{equation*}
\left\vert H_{\varepsilon }\left( x\right) -H_{\varepsilon }\left( y\right)
\right\vert \leq \left\Vert B\left( \varepsilon \right) \right\Vert
\left\vert x-y\right\vert \quad \text{\emph{for }}x\text{, }y\in \mathbb{R}%
^{N}\text{, }\varepsilon \in E\text{,}
\end{equation*}%
\emph{where}%
\begin{equation*}
\left\Vert B\left( \varepsilon \right) \right\Vert =\left( \sum_{i,j=1}^{N} 
\left[ b_{ij}\left( \varepsilon \right) \right] ^{2}\right) ^{\frac{1}{2}}%
\text{.}
\end{equation*}%
\emph{Therefore the desired result follows.}
\end{example}

\begin{example}
\emph{\label{ex2.6} Let }$P:\mathbb{R}^{N}\rightarrow \mathbb{R}^{N}$ \emph{%
be a linear transformation. For} $\varepsilon \in E=\mathbb{R}$, \emph{we
define}%
\begin{equation*}
\exp \left( -\varepsilon P\right) =\sum_{n=0}^{+\infty }\frac{\left(
-1\right) ^{n}}{n!}\varepsilon ^{n}P^{n}\text{.}
\end{equation*}%
\emph{Next, let }$k\in \mathbb{R}$ \emph{with} $k>\left\Vert P\right\Vert
=\sup_{\left\vert x\right\vert \leq 1}\left\vert Px\right\vert $. \emph{We
define} $\mathcal{H=}\left( H_{\varepsilon }\right) _{\varepsilon \in 
\mathbb{R}}$ \emph{as} $H_{\varepsilon }\left( x\right) =\exp \left(
-k\varepsilon \right) \exp \left( -\varepsilon P\right) x$ \emph{for} $x\in 
\mathbb{R}^{N}$. \emph{Each} $H_{\varepsilon }$ \emph{is an automorphism of} 
$\mathbb{R}^{N}$ \emph{and for fixed} $x\in \mathbb{R}^{N}$, \emph{the map }$%
\varepsilon \rightarrow H_{\varepsilon }\left( x\right) $\emph{\ sends
continuously }$\mathbb{R}$ \emph{into} $\mathbb{R}^{N}$. \emph{Finally, we
have} $\left\vert H_{\varepsilon }\left( x\right) \right\vert \leq
e^{-\left( k-\left\Vert P\right\Vert \right) \varepsilon }\left\vert
x\right\vert $ \emph{for} $x\in \mathbb{R}^{N}$, $\varepsilon \in \mathbb{R}$%
. \emph{Hence it follows that the one parameter group} $\mathcal{H}$ \emph{%
meets the requirements of Example \ref{ex2.5}, and is therefore an
absorptive continuous} \emph{action of} $\mathbb{R}$ \emph{on} $\mathbb{R}%
^{N}$.
\end{example}

\section{HOMOGENEOUS MEASURES}

Throughout the present section, $E$ denotes an $\mathbb{R}$-group (with $e$
and $\theta $ defined as in subsection 2.1), $X$ denotes a locally compact
space not reduced to one point, and $\mathcal{H=}\left( H_{\varepsilon
}\right) _{\varepsilon \in E}$ denotes an absorptive continuous action of $E$
on $X$ with center $\omega $.

Let us observe that, according to Proposition \ref{pr2.2}, the measure $%
\delta _{\omega }$ (Dirac measure at $\omega $) is invariant for $\mathcal{H}
$, i.e., $H_{\varepsilon }\left( \delta _{\omega }\right) =\delta _{\omega }$
for $\varepsilon \in E$. One may naturally question whether there exist
other invariant positive measures on $X$. The next proposition will allow us
to answer this question. Let us make a definition beforehand.

\begin{definition}
\emph{\label{def3.1}}\ \emph{A Radon measure }$\lambda $\emph{\ on }$X$\emph{%
\ is termed }nontrivial\emph{\ if }$\lambda $\emph{\ is distinct from both }$%
\delta _{\omega }$\emph{\ and }$0$\emph{\ (the zero measure on }$X$\emph{).}
\end{definition}

Finally, before turning to the statement and proof of the alleged
proposition, it is worth recalling that the notion of a homogeneous measure
on $X$ is defined in Section 1.

\begin{proposition}
\label{pr3.1} Let $\lambda $ be a nontrivial positive Radon measure on $X$.
Suppose $\lambda $ is homogeneous (for $\mathcal{H}$). Then there exists a
continuous group homomorphism $h:E\rightarrow \mathbb{R}_{+}^{\ast }$ such
that:%
\begin{equation}
H_{\varepsilon }\left( \lambda \right) =h\left( \varepsilon ^{-1}\right)
\lambda \text{\qquad }\left( \varepsilon \in E\right) \text{,}  \label{eq3.1}
\end{equation}%
\begin{equation}
\lim_{\varepsilon \rightarrow \theta }h\left( \varepsilon ^{-1}\right) =0%
\text{.}  \label{eq3.2}
\end{equation}%
On the other hand, we have that%
\begin{equation}
\lambda \left( \left\{ \omega \right\} \right) =0\text{.}  \label{eq3.3}
\end{equation}
\end{proposition}

\begin{proof}
First of all, it is trivial that for fixed $\varepsilon \in E$, the number $%
c\left( \varepsilon \right) $ in (H) (see Section 1) is unique. This yields
a map $\varepsilon \rightarrow c\left( \varepsilon \right) $ of $E$ into $%
\mathbb{R}_{+}^{\ast }$. This map is a homomorphism, i.e., $c\left(
\varepsilon \varepsilon ^{\prime }\right) =c\left( \varepsilon \right)
c\left( \varepsilon ^{\prime }\right) $ for $\varepsilon ,\varepsilon
^{\prime }\mathcal{\in }E$, as is immediate by the equality $H_{\varepsilon
\varepsilon ^{\prime }}\left( \lambda \right) =H_{\varepsilon ^{\prime
}}\left( H_{\varepsilon }\left( \lambda \right) \right) $ (see \cite{bib4},
p.72). Let us show that this homomorphism is continuous. It is enough to
check the continuity at $\varepsilon =e$. For this purpose, fix $\varphi \in 
\mathcal{K}\left( X\right) $ with $\lambda \left( \varphi \right) =1$ (such
a $\varphi $ exists because $\lambda $ is nonzero). Then%
\begin{equation}
c\left( \varepsilon \right) =\int \varphi \left( H_{\varepsilon }\left(
x\right) \right) d\lambda \left( x\right) \qquad \left( \varepsilon \in
E\right) \text{.}  \label{eq3.4}
\end{equation}%
Now, let $F$ be an open elementary set containing the support of $\varphi $
(use corollary \ref{co2.3} and observe that the set $H_{n^{-1}}\left(
F\right) $ therein is elementary) and let $\alpha \in E$ with $\alpha >e$.
The set $I=(\theta ,\alpha ]\cap E$ is a neighbourhood of $e$ in $E$ and we
have $H_{\varepsilon ^{-1}}\left( F\right) \subset H_{\alpha ^{-1}}\left(
F\right) $ for $\varepsilon \in I$. It follows that $Supp\left( \varphi
\circ H_{\varepsilon }\right) \subset H_{\alpha ^{-1}}\left( F\right) $ for $%
\varepsilon \in I$ ($Supp$ stands for "support"). Hence%
\begin{equation}
\left\vert \varphi \left( H_{\varepsilon }\left( x\right) \right)
\right\vert \leq \left\Vert \varphi \right\Vert _{\infty }f\left( x\right)
\qquad \left( x\in X\text{, }\varepsilon \in I\right) \text{,}  \label{eq3.5}
\end{equation}%
where $f$ is the characteristic function of $H_{\alpha ^{-1}}\left( F\right) 
$ in $X$. Therefore the claimed continuity follows by a classical argument
(see, e.g., \cite{bib5}, p.144). The result is that the map $h:E\rightarrow 
\mathbb{R}_{+}^{\ast }$ given by $h\left( \varepsilon \right) =c\left(
\varepsilon ^{-1}\right) $, $\varepsilon \in E$, is a continuous
homomorphism and further (\ref{eq3.1}) holds. The next point is to check (%
\ref{eq3.2}), that is, $\lim_{\varepsilon \rightarrow \theta }c\left(
\varepsilon \right) =0$. To this end, let $\varphi \in \mathcal{K}\left(
X\right) $ with $\varphi \left( \omega \right) =0$ and $\lambda \left(
\varphi \right) =1$ (such a $\varphi $ does exist because $\lambda $ is
nontrivial; use \cite{bib5}, p.43, Lemme 1). For any $x\in X$ ($x=\omega $
included!), we have $\varphi \left( H_{\varepsilon }\left( x\right) \right)
\rightarrow 0$ as $\varepsilon \rightarrow \theta $, as is straightforward
by Corollary \ref{co2.2} and use of the fact that $\varphi $ has a compact
support. Hence, by (\ref{eq3.4})-(\ref{eq3.5}) and use of the dominated
convergence theorem, it follows that $\lim_{\varepsilon \rightarrow \theta
}c\left( \varepsilon \right) =\lim_{\varepsilon \rightarrow \theta }h\left(
\varepsilon ^{-1}\right) =0$. It remains to check that (\ref{eq3.3}) holds.
Let $F$ be a compact elementary set. In view of part (iii) of Proposition %
\ref{pr2.3}, we have ${\large \cap }_{n=1}^{+\infty }H_{n}\left( F\right)
=\left\{ \omega \right\} $. Combining this with the relation $H_{n+1}\left(
F\right) \subset H_{n}\left( F\right) $ and using a classical argument from
integration theory we arrive at $\lambda \left( \left\{ \omega \right\}
\right) =\lim_{n\rightarrow +\infty }\lambda \left( H_{n}\left( F\right)
\right) $. Hence the claimed result follows by the equality $\lambda \left(
H_{n}\left( F\right) \right) =h\left( n\right) \lambda \left( F\right) $ and
use of the fact that $h\left( n\right) \rightarrow 0$ as $n\rightarrow
+\infty $ (this is straightforward from (\ref{eq3.2})).
\end{proof}

As a direct consequence of this, there is the following corollary.

\begin{corollary}
\label{co3.1} There exists no nontrivial invariant positive measure on $X$
(for $\mathcal{H}$).
\end{corollary}

The next result is about the existence of nontrivial homogeneous positive
measures.

\begin{theorem}
\label{th3.1} Suppose each point in $X$ has a countable base of
neighbourhoods. Then, there is always some nontrivial homogeneous positive
Radon measure on $X$. More precisely, to each group homomorphism $%
h:E\rightarrow \mathbb{R}_{+}^{\ast }$ meeting the requirements of \emph{(RG)%
}$_{3}$ (in Definition \ref{def2.1}), there is attached a nontrivial
positive Radon measure $\lambda $ on $X$ satisfying (\ref{eq3.1}).
\end{theorem}

\begin{proof}
Let us fix a homomorphism $h:E\rightarrow \mathbb{R}_{+}^{\ast }$ as stated
above. Let $\nu =h{\small \cdot }m$, where $m$ denotes Haar measure on $E$.
Thus, $\nu $ is the Radon measure on $E$ given by $\nu \left( \varphi
\right) =\int \varphi \left( \varepsilon \right) h\left( \varepsilon \right)
dm\left( \varepsilon \right) $ for $\varphi \in \mathcal{K}\left( E\right) $%
. Now, fix a nonzero positive Radon measure $\mu $ on $X$ with compact
support $S$ contained in the open set $X^{\ast }=X\backslash \left\{ \omega
\right\} $ (e.g., $\mu $ is the Dirac measure at some point $a\in X^{\ast }$%
). On the other hand, fix freely some $x\in X^{\ast }$ and denote by $G_{x}$
the map of $E$ into $X$ given by $G_{x}\left( \varepsilon \right)
=H_{\varepsilon }\left( x\right) $ $\left( \varepsilon \in E\right) $. Our
preliminary task is to check that this map is $\nu $-proper (i.e., $G_{x}$
is $\nu $-measurable and further the inverse image, $G_{x}^{-1}\left(
K\right) $, of any compact set $K\subset X$ is $\nu $-integrable) so we can
define the image measure $G_{x}\left( \nu \right) $ on $X$ (see, e.g., \cite%
{bib4}, p.69). To begin, note that $G_{x}$ is $\nu $-measurable, since it is
continuous. Next, let $K$ be any compact set in $X$. According to Corollary %
\ref{co2.2}, we may consider some $\alpha \in E$ such that $H_{\varepsilon
}\left( x\right) \in X\backslash K$ for $\varepsilon <\alpha $. Then, we
have $G_{x}^{-1}\left( K\right) \subset E_{\alpha }=\left\{ \varepsilon \in
E:\varepsilon \geq \alpha \right\} $. Hence the integrability of $%
G_{x}^{-1}\left( K\right) $ follows from that of $E_{\alpha }$ (see (RG)$%
_{3} $). Consequently we can define $\lambda _{x}=G_{x}\left( \nu \right) $,
i.e., $\lambda _{x}$ is the Radon measure on $X$ given by 
\begin{equation*}
\lambda _{x}\left( \varphi \right) =\int \varphi \left( H_{\varepsilon
}\left( x\right) \right) d\nu \left( \varepsilon \right) \qquad \quad \left(
\varphi \in \mathcal{K}\left( X\right) \right)
\end{equation*}%
and that for any arbitrarily fixed $x\in X^{\ast }$. The measure $\lambda
_{x}$ is nonzero (indeed, according to \cite{bib4}, p.70, Corollaire 4, the
support of $\lambda _{x}$ is precisely the closure of $G_{x}\left( E\right) $
in $X$) and manifestly positive.

At the present time, let $\mathcal{M}_{+}\left( X\right) $ stand for the
convex cone of all positive Radon measures on $X$, $\mathcal{M}_{+}\left(
X\right) $ provided with the relative weak $\ast $ topology on $\mathcal{M}%
\left( X\right) $ (space of all complex Radon measures on $X$). We introduce
the mapping $\Lambda :X\rightarrow \mathcal{M}_{+}\left( X\right) $ given by%
\begin{equation*}
\Lambda \left( \omega \right) =\delta _{\omega }\text{\ and }\Lambda \left(
x\right) =\lambda _{x}\text{ if }x\neq \omega \text{,}
\end{equation*}%
where $\delta _{\omega }$ is Dirac measure (on $X$) at $\omega $. Our
purpose is the following: Firstly we show that $\Lambda $ is $\mu $%
-integrable (for integration of positive measures see \cite{bib4}), i.e.,
for each fixed $\varphi \in \mathcal{K}\left( X\right) $,the complex
function $x\rightarrow \Lambda \left( x\right) \left( \varphi \right) $ on $X
$ is $\mu $-integrable. Secondly we show that the integral of $\Lambda $ for 
$\mu $, namely the measure $\lambda =\int \Lambda \left( x\right) d\mu
\left( x\right) $ on $X$, has the required properties. This will be
accomplished in two steps.

\textbf{Step1}. The aim here is to show that $\Lambda $ is $\mu $%
-integrable. So, let $\varphi \in \mathcal{K}\left( X\right) $ be
arbitrarily fixed. We begin by verifying that the map $x\rightarrow \Lambda
\left( x\right) \left( \varphi \right) $ of $X$ into $\mathbb{C}$ is $\mu $%
-measurable. Noting that the set $\left\{ \omega \right\} $ is $\mu $%
-negligible, we see that it is enough to check that the map $x\rightarrow
\lambda _{x}\left( \varphi \right) $ of $X^{\ast }$ into $\mathbb{C}$ is
continuous. We will need the following property:\bigskip

(P) To any given compact set $K\subset X^{\ast }$ there is attached some $%
\alpha \in E$ such that $\left\vert \varphi \left( H_{\varepsilon }\left(
x\right) \right) \right\vert \leq \left\Vert \varphi \right\Vert _{\infty
}f_{\alpha }\left( \varepsilon \right) $ for $\varepsilon \in E$, $x\in K$,
where $f_{\alpha }$ denotes the characteristic function (in $E$) of $%
E_{\alpha }=\left\{ \varepsilon \in E:\varepsilon \geq \alpha \right\} $%
.\bigskip

\noindent To establish (P), we introduce a compact elementary set $F$ such
that $Supp\varphi \subset F$ (use Corollary \ref{co2.3}). Put $U=X\backslash
F$ and bear in mind that $\varepsilon \leq \varepsilon ^{\prime }$ implies $%
H_{\varepsilon }\left( U\right) \subset H_{\varepsilon ^{\prime }}\left(
U\right) $. Now, fix freely any compact set $K\subset X^{\ast }$. By
assigning to $x\in K$ some $\alpha _{x}\in E$ such that $x\in H_{\alpha
_{x}^{-1}}\left( U\right) $ (use Corollary \ref{co2.2}), we get an open
covering $\left\{ H_{\alpha _{x}^{-1}}\left( U\right) \right\} _{x\in K}$ of 
$K$, from which we extract a finite family $\left\{ H_{\alpha
_{i}^{-1}}\left( U\right) \right\} _{1\leq i\leq n}$ covering $K$. It
follows that $K\subset H_{\alpha ^{-1}}\left( U\right) $ with $\alpha
=\min_{1\leq i\leq n}\alpha _{i}$, hence $H_{\varepsilon }\left( K\right)
\subset U$ for $\varepsilon \mathfrak{<}\alpha $. We deduce that $\varphi
\left( H_{\varepsilon }\left( x\right) \right) =0$ for $x\in K$ and $%
\varepsilon \mathfrak{<}\alpha $, from which (P) follows.

Having made this point, the continuity of the function $x\rightarrow \lambda
_{x}\left( \varphi \right) $ $\left( x\in X^{\ast }\right) $ at any
arbitrarily given $a\in X^{\ast }$ follows by choosing in (P) the compact
set $K$ as being a neighbourhood of $a$ and then applying a classical
argument (see \cite{bib5}, p.144) to the mapping $\left( \varepsilon
,x\right) \rightarrow \varphi \left( H_{\varepsilon }\left( x\right) \right) 
$ of $E\times X$ into $\mathbb{C}$. This shows the $\mu $-measurability of
the function $x\rightarrow \Lambda \left( x\right) \left( \varphi \right) $%
\quad $\left( x\in X\right) $. Thus, to conclude that this function is $\mu $%
-integrable it only remains to check that some nonnegative $\mu $-integrable
\ function $\psi :X\rightarrow \mathbb{R}$ exists such that $\left\vert
\Lambda \left( x\right) \left( \varphi \right) \right\vert \leq \psi \left(
x\right) $ for $\mu $-almost all $x\in X$, or equivalently (see, e.g., \cite%
{bib5}, p.156) such that $\left\vert \mathcal{\chi }_{S}\left( x\right)
\Lambda \left( x\right) \left( \varphi \right) \right\vert \leq \psi \left(
x\right) $ for $\mu $-almost all $x\in X$, where $\mathcal{\chi }_{S}$ is
the characteristic function\ of $S$ (the support of $\mu $). But this is
straightforward by (P). Indeed, choosing in (P) the particular compact set $%
K=S$ yields some $\alpha \mathfrak{\in }E$ such that $\left\vert \mathcal{%
\chi }_{S}\left( x\right) \lambda _{x}\left( \varphi \right) \right\vert
\leq \left\Vert \varphi \right\Vert _{\infty }\nu \left( E_{\alpha }\right) 
\mathcal{\chi }_{S}\left( x\right) $ for all $x\in X$ with $x\neq \omega $.
This completes Step1.

\textbf{Step2}. According to Step1, we may put%
\begin{equation*}
\lambda =\int \lambda _{x}d\mu \left( x\right) \text{.}
\end{equation*}%
Specifically, $\lambda $ is the positive Radon measure on $X$ given by $%
\lambda \left( \varphi \right) =\int \lambda _{x}\left( \varphi \right) d\mu
\left( x\right) $\quad $\left( \varphi \in \mathcal{K}\left( X\right)
\right) $, or more explicitly (see \cite{bib4}, p.17) by 
\begin{equation*}
\lambda \left( \varphi \right) =\int d\mu \left( x\right) \int \varphi
\left( y\right) d\lambda _{x}\left( y\right) \qquad \left( \varphi \in 
\mathcal{K}\left( X\right) \right) \text{.}
\end{equation*}%
Our purpose in the present step is to verify that $\lambda $ has the
required properties. The first point will be to verify that $\lambda $ is
nonzero. Let $\varphi \in \mathcal{K}\left( X\right) $ with $\varphi \geq 0$
and $\varphi =1$ on $S$ ($\varphi $ exists by Urysohn's lemma; see also \cite%
{bib5}, p.43, Lemma 1). We claim that $\lambda \left( \varphi \right) \neq 0$%
. Indeed, assuming the contrary leads to $\int \lambda _{x}\left( \varphi
\right) \psi \left( x\right) d\mu \left( x\right) =\lambda \left( \varphi
\right) =0$, where $\psi \in \mathcal{K}\left( X\right) $, $\psi \geq 0$, $%
\psi =1$ in a neighbourhood of $S$, $\psi $ having support in $X^{\ast }$.
Consequently $\lambda _{x}\left( \varphi \right) \psi \left( x\right) =0$
for any $x\in S$ (use \cite{bib5}, p.69, Proposition 9). Hence $\lambda
_{x}\left( \varphi \right) =0$ for $x\in S$. We deduce that $\varphi \left(
H_{\varepsilon }\left( x\right) \right) =0$ for $x\in S$ and $\varepsilon
\in E$ (note that the support of $\nu $ is the whole $E$). Therefore $%
\varphi =0$ on $S$, a contradiction and so $\lambda $ is a nonzero positive
measure. The next point is to establish (\ref{eq3.1}). To this end fix
freely some $s\in E$ and begin by recalling that the translate $\tau _{s}\nu 
$ is defined to be the Radon measure on $E$ given by $\tau _{s}\nu \left(
f\right) =\int f\left( s\varepsilon \right) d\nu \left( \varepsilon \right) $%
\quad $\left( f\in \mathcal{K}\left( E\right) \right) $, and by bearing the
equation $\tau _{s}\nu =h\left( s^{-1}\right) \nu $ in mind (this is
immediate by the translation invariance of $m$). Then, given any arbitrary $%
\varphi \in \mathcal{K}\left( X\right) $, one has 
\begin{equation*}
H_{s}\left( \lambda \right) \left( \varphi \right) =\int \varphi \left(
H_{s}\left( x\right) \right) d\lambda \left( x\right)
\end{equation*}%
\begin{equation*}
\qquad \qquad \qquad \qquad \quad =\int d\mu \left( x\right) \int \varphi
\left( H_{s}\left( y\right) \right) d\lambda _{x}\left( y\right)
\end{equation*}%
\begin{equation*}
\qquad \qquad \qquad \qquad =\int d\mu \left( x\right) \int \varphi \left(
H_{s\varepsilon }\left( x\right) \right) d\nu \left( \varepsilon \right)
\end{equation*}%
\begin{equation*}
\qquad \qquad \qquad \qquad \qquad \quad =h\left( s^{-1}\right) \int d\mu
\left( x\right) \int \varphi \left( H_{\varepsilon }\left( x\right) \right)
d\nu \left( \varepsilon \right)
\end{equation*}%
\begin{equation*}
\qquad \qquad \qquad =h\left( s^{-1}\right) \int \lambda _{x}\left( \varphi
\right) d\mu \left( x\right)
\end{equation*}%
\begin{equation*}
\qquad \qquad =h\left( s^{-1}\right) \int \varphi d\lambda \text{.}
\end{equation*}%
Hence (\ref{eq3.1}) follows. Finally it is clear that $\lambda \neq \delta
_{\omega }$, since $h$ is not the constant homomorphism $\mathcal{\chi }%
:E\rightarrow \mathbb{R}_{+}^{\ast }$ (viz. $\mathcal{\chi }\left(
\varepsilon \right) =1$ for $\varepsilon \in E$). The theorem is proved.
\end{proof}

We are now able to set the following definition.

\begin{definition}
\emph{\label{def3.2} By a }homogenizer\emph{\ we will mean a triple} $\left(
X,\mathcal{H},\lambda \right) $ \emph{in which:}

\emph{(i) }$X$\emph{\ is a noncompact locally compact space and further each
point in }$X$\emph{\ has a countable base of neighbourhoods;}

\emph{(ii)} $\mathcal{H}$ \emph{is an absorptive continuous} $\mathbb{R}$%
\emph{-group action on }$X$\emph{;}

\emph{(iii) }$\lambda $\emph{\ is a homogeneous nontrivial positive Radon\
measure\ on }$X$\emph{.}
\end{definition}

It is worth recalling that if $X$ is as above, then, on one hand its
topology is never the discrete one, on the other hand $X$ is $\sigma $%
-compact (see Corollary \ref{co2.1}).

\begin{remark}
\emph{\label{rem3.1} Given a homogenizer }$\left( X,\mathcal{H},\lambda
\right) $, \emph{we shall always assume that }$X$\emph{\ is equipped with\
the measure }$\lambda $\emph{.}
\end{remark}

\section{MEAN VALUE}

Throughout the present section, $\left( X,\mathcal{H},\lambda \right) $ is a
given homogenizer with $\mathcal{H=}\left( H_{\varepsilon }\right)
_{\varepsilon \in E}$, where $E$ is the acting $\mathbb{R}$-group (with $e$
and $\theta $ as before; see subsection 2.1). On the other hand, to the
measure $\lambda $ there is attached a (unique) continuous homomorphism $%
h:E\rightarrow \mathbb{R}_{+}^{\ast }$ such that (\ref{eq3.1})-(\ref{eq3.2})
hold true.

Before embarking upon discussing the mean value on $\left( X,\mathcal{H}%
,\lambda \right) $, let us first draw attention to what is generically meant
by a mean value on a topological space. Let $\mathbf{T}$ be a Hausdorff
topological space, and let $\mathcal{B}\left( \mathbf{T}\right) $ denote the
space of bounded continuous complex functions on $\mathbf{T}$. We provide $%
\mathcal{B}\left( \mathbf{T}\right) $ with the supremum norm, viz. $%
\left\Vert u\right\Vert _{\infty }=\sup_{t\in \mathbf{T}}\left\vert u\left(
t\right) \right\vert $\quad $\left( u\in \mathcal{B}\left( \mathbf{T}\right)
\right) $, which makes it a Banach space.

\begin{definition}
\emph{\label{def4.1} A }mean value\emph{\ on }$\mathbf{T}$ \emph{is defined
to be an unbounded linear operator }$m$ \emph{from} $\mathcal{B}\left( 
\mathbf{T}\right) $ \emph{to} $\mathbb{C}$ \emph{with the following
properties:}

\emph{(MV)}$_{1}$ \emph{The domain, }$D\left( m\right) $\emph{, of }$m$\emph{%
\ is a closed vector subspace of }$\mathcal{B}\left( \mathbf{T}\right) $ 
\emph{containing the constants.}

\emph{(MV)}$_{2}$\emph{\ We have}

\emph{(i) }$m\left( f\right) \geq 0$\emph{\ for }$f\in D\left( m\right) $%
\emph{\ with }$f\geq 0$\emph{,}

\emph{(ii) }$m\left( 1\right) =1$\emph{.}
\end{definition}

There is no difficulty in verifying that a mean value $m$ on $\mathbf{T}$
satisfies the inequality $\left\vert m\left( f\right) \right\vert \leq
\left\Vert f\right\Vert _{\infty }$ for $f\in D\left( m\right) $ (proceed as
in \cite{bib20}, Proposition \ref{pr2.1}). Likewise it is an easy exercise
to give various examples of mean values (see, e.g., \cite{bib20}).

We turn now to a more specific case. Let $\Pi ^{\infty }\left( X,\mathcal{H}%
,\lambda \right) $ be the space of those functions $u\in \mathcal{B}\left(
X\right) $ for which a complex number $\widetilde{u}$ exists such that $%
u\circ H_{\varepsilon }\rightarrow \widetilde{u}$ in $L^{\infty }\left(
X\right) $-weak $\ast $ as $\varepsilon \rightarrow \theta $, where we still
call $\widetilde{u}$ the constant function $f\in \mathcal{B}\left( X\right) $%
, $f\left( x\right) =\widetilde{u}$\quad $\left( x\in X\right) $. It is an
easy matter to check that $\Pi ^{\infty }\left( X,\mathcal{H},\lambda
\right) $ is a closed vector subspace of $\mathcal{B}\left( X\right) $
(equipped with the supremum norm).

\begin{definition}
\emph{\label{def4.2}} \emph{By the mean value on} $\left( X,\mathcal{H}%
,\lambda \right) $ \emph{is meant the unbounded }\ \emph{linear operator }$M$%
\emph{\ from} $\mathcal{B}\left( X\right) $ \emph{to} $\mathbb{C}$, \emph{%
whose domain is} $D\left( M\right) =\Pi ^{\infty }\left( X,\mathcal{H}%
,\lambda \right) $ \emph{and whose value}\ \emph{at} $u\in D\left( M\right) $
\emph{is} $M\left( u\right) =\widetilde{u}$ \emph{(the above weak limit).}
\end{definition}

The mean value on $\left( X,\mathcal{H},\lambda \right) $ verifies (MV)$_{1}$%
-(MV)$_{2}$ (this is an easy exercise) and so Definition \ref{def4.2} is
justified. Furthermore, as will be clarified in the next section, the mean
value on $\left( X,\mathcal{H},\lambda \right) $ generalizes the mean value
for ponderable functions introduced earlier on numerical spaces (see \cite%
{bib20}). On the other hand, it extends various notions of mean value
available in mathematical analysis. Let us exhibit two such examples.

\begin{example}
\emph{\label{ex4.1} Let }$\mathcal{B}_{\infty }\left( X\right) $ \emph{be
the space of continuous complex\ functions} $u$\emph{\ on }$X$\emph{\ such
that} $\lim_{x\rightarrow \infty }u\left( x\right) =\xi \in \mathbb{C}$, 
\emph{where }$\infty $\emph{\ is the point at infinity of the Alexandroff} 
\emph{compactification of }$X$. $\mathcal{B}_{\infty }\left( X\right) $ 
\emph{is a closed vector subspace of} $\mathcal{B}\left( X\right) $, \emph{%
and the (unbounded)} \emph{linear operator }$m$\emph{\ from} $\mathcal{B}%
\left( X\right) $ \emph{to} $\mathbb{C}$ \emph{defined by} $D\left( m\right)
=\mathcal{B}_{\infty }\left( X\right) $ \emph{and} $m\left( u\right)
=\lim_{x\rightarrow \infty }u\left( x\right) $ \emph{for} $u\in D\left(
m\right) $, \emph{is a mean value on }$X$\emph{\ called the mean value for} $%
\mathcal{B}_{\infty }\left( X\right) $. \emph{This being so, it is} \emph{%
immediate by dominated convergence (use (\ref{eq3.3}) and Corollary \ref%
{co2.2}) that} $\mathcal{B}_{\infty }\left( X\right) \subset \Pi ^{\infty
}\left( X,\mathcal{H},\lambda \right) $ \emph{and} $M\left( u\right)
=m\left( u\right) $ \emph{for} $u\in D\left( m\right) $, \emph{and so} $M$ 
\emph{is an extension of }$m$\emph{.}
\end{example}

\begin{example}
\emph{\label{ex4.2} Take }$X=\mathbb{R}^{N}$ $\left( N\geq 1\right) $, $%
\mathcal{H=}\left( H_{\varepsilon }\right) _{\varepsilon >0}$ \emph{with} $%
H_{\varepsilon }\left( x\right) =\frac{x}{\varepsilon }$ \emph{for} $x\in 
\mathbb{R}^{N}$ \emph{(see Example \ref{ex2.4}), and }$\lambda $\emph{\ as
being Lebesgue measure on} $\mathbb{R}^{N}$. \emph{It is a classical fact
that }$H_{\varepsilon }\left( \lambda \right) =\varepsilon ^{N}\lambda $%
\emph{\ for }$\varepsilon >0$\emph{, so that the triple} $\left( \mathbb{R}%
^{N},\mathcal{H},\lambda \right) $\emph{\ is a homogenizer. Now, putting }$%
Y=\left( 0,1\right) ^{N}$\emph{, let} $\mathcal{C}_{per}\left( Y\right) $ 
\emph{be the space of those continuous}\ \emph{complex functions }$u$\emph{\
on} $\mathbb{R}^{N}$ \emph{that are }$Y$\emph{-periodic, i.e., that satisfy} 
$u\left( y+k\right) =u\left( y\right) $ \emph{for} $y\in \mathbb{R}^{N}$ 
\emph{and} $k\in \mathbb{Z}^{N}$ \emph{(}$\mathbb{Z}$ \emph{denotes the
integers).} $\mathcal{C}_{per}\left( Y\right) $ \emph{is a closed vector
subspace of} $\mathcal{B}\left( \mathbb{R}^{N}\right) $,\emph{\ and the
(unbounded) linear operator }$m$\emph{\ from }$\mathcal{B}\left( \mathbb{R}%
^{N}\right) $ \emph{to} $\mathbb{C}$ \emph{defined by} $D\left( m\right) =%
\mathcal{C}_{per}\left( Y\right) $ \emph{and} $m\left( u\right)
=\int_{Y}u\left( y\right) dy$\quad $\left( u\in \mathcal{C}_{per}\left(
Y\right) \right) $, \emph{is a mean value on} $\mathbb{R}^{N}$. \emph{%
Furthermore, it is\ well-known that} $\mathcal{C}_{per}\left( Y\right)
\subset \Pi ^{\infty }\left( \mathbb{R}^{N},\mathcal{H},\lambda \right) $ 
\emph{and} $m\left( u\right) =M\left( u\right) $ \emph{for} $u\in D\left(
m\right) $ \emph{(see \cite{bib3}).} \emph{Thus, the mean} \emph{value on} $%
\left( \mathbb{R}^{N},\mathcal{H},\lambda \right) $ \emph{extends the mean
value for }$Y$\emph{-periodic functions on }$\mathbb{R}^{N}$.
\end{example}

Thus, in all\ probability, the mean value on $\left( X,\mathcal{H},\lambda
\right) $ is a useful tool for the study of the asymptotic properties of the
action $\mathcal{H}$. To see that the mean value on $\left( X,\mathcal{H}%
,\lambda \right) $ is actually necessary for this purpose, let us fix freely 
$u\in L^{1}\left( X\right) $. According to (\ref{eq3.1}), we have%
\begin{equation*}
\int u\left( H_{\varepsilon }\left( x\right) \right) d\lambda \left(
x\right) =h\left( \varepsilon ^{-1}\right) \int u\left( x\right) d\lambda
\left( x\right) \text{,}
\end{equation*}%
hence, by (\ref{eq3.2}),%
\begin{equation*}
\lim_{\varepsilon \rightarrow \theta }\int u\left( H_{\varepsilon }\left(
x\right) \right) d\lambda \left( x\right) =0\text{.}
\end{equation*}%
More generally, for $u\in L^{p}\left( X\right) $\quad $\left( 1\leq
p<+\infty \right) $, it is immediate that

\noindent $\left\Vert u\circ H_{\varepsilon }\right\Vert _{L^{p}\left(
X\right) }\rightarrow 0$ as $\varepsilon \rightarrow \theta $. Therefore, to
have significative information about the asymptotic properties of $\mathcal{H%
}$ we had the need to introduce an alternative notion to integration, namely
the mean value on $\left( X,\mathcal{H},\lambda \right) $.

However, in order to know more about the asymptotic properties of $\mathcal{H%
}$, we need to extend $M$ beyond $\Pi ^{\infty }\left( X,\mathcal{H},\lambda
\right) $.

Let us fix freely a real $p\geq 1$. We first of all introduce the space $\Xi
^{p}\left( X,\mathcal{H},\lambda \right) $ of those\ functions $u\in
L_{loc}^{p}\left( X\right) $ for which the sequence $\left( u\circ
H_{\varepsilon }\right) _{\varepsilon \leq e}$ is bounded in $%
L_{loc}^{p}\left( X\right) $. This is a vector subspace of $%
L_{loc}^{p}\left( X\right) $. Furthermore, let $B_{\mathcal{H}}$ be an open
elementary set in $X$ (Definition \ref{def2.4}). Let 
\begin{equation}
\left\Vert u\right\Vert _{\Xi ^{p}}=\sup_{\varepsilon \leq e}\left( \int_{B_{%
\mathcal{H}}}\left\vert u\left( H_{\varepsilon }\left( x\right) \right)
\right\vert ^{p}d\lambda \left( x\right) \right) ^{\frac{1}{p}}
\label{eq4.1}
\end{equation}%
for $u\in \Xi ^{p}\left( X,\mathcal{H},\lambda \right) $, which defines a
seminorm on $\Xi ^{p}\left( X,\mathcal{H},\lambda \right) $.

\begin{lemma}
\label{lem4.1} The preceding seminorm is actually a norm under which $\Xi
^{p}\left( X,\mathcal{H},\lambda \right) $ is a Banach space.
\end{lemma}

\begin{proof}
For each compact set $K\subset X$, put 
\begin{equation*}
q_{K}\left( u\right) =\sup_{\varepsilon \leq e}\left( \int_{K}\left\vert
u\left( H_{\varepsilon }\left( x\right) \right) \right\vert ^{p}d\lambda
\left( x\right) \right) ^{\frac{1}{p}}\qquad \quad \left( u\in \Xi
^{p}\left( X,\mathcal{H},\lambda \right) \right) \text{,}
\end{equation*}%
which gives a family of seminorms $\left\{ q_{K}\right\} $ ($K$ ranging over
the compact sets in $X$) on $\Xi ^{p}\left( X,\mathcal{H},\lambda \right) $.
The natural topology on $\Xi ^{p}\left( X,\mathcal{H},\lambda \right) $ is
that defined by this family of seminorms. So topologized, $\Xi ^{p}\left( X,%
\mathcal{H},\lambda \right) $ is a separated locally convex space. Moreover,
because $X$ is $\sigma $-compact, this topology may be defined by a
countable family of seminorms as above, so that there is a countable base of
neighbourhoods at the origin in $\Xi ^{p}\left( X,\mathcal{H},\lambda
\right) $. The result is that the proof of the completeness of $\Xi
^{p}\left( X,\mathcal{H},\lambda \right) $ is equivalent to that of the
sequential completeness. But then the sequential completeness of $\Xi
^{p}\left( X,\mathcal{H},\lambda \right) $ follows by a classical method
using the facts that $L_{loc}^{p}\left( X\right) $ is complete, $\Xi
^{p}\left( X,\mathcal{H},\lambda \right) $ is continuously embedded in $%
L_{loc}^{p}\left( X\right) $, and for fixed $\varepsilon \in E$, the
transformation $u\rightarrow u\circ H_{\varepsilon }$\ maps continuously $%
L_{loc}^{p}\left( X\right) $ into itself. Thus, the lemma is proved if we
can check that the (natural) topology on $\Xi ^{p}\left( X,\mathcal{H}%
,\lambda \right) $ is equivalent to that defined by the seminorm in (\ref%
{eq4.1}), which seminorm will then turn out to be a complete norm. However,
a little moment of reflexion reveals that the whole problem reduces to
showing that to each compact set $K\subset X$ corresponds some constant $%
c=c\left( K,\mathcal{H}\right) >0$ such that%
\begin{equation}
\left( \int_{K}\left\vert u\left( H_{\varepsilon }\left( x\right) \right)
\right\vert ^{p}d\lambda \left( x\right) \right) ^{\frac{1}{p}}\leq
c\left\Vert u\right\Vert _{\Xi ^{p}}  \label{eq4.2}
\end{equation}%
for all $u\in \Xi ^{p}\left( X,\mathcal{H},\lambda \right) $ and all $%
\varepsilon \in E$ with $\varepsilon \leq e$. So, let $K$ be a compact set
in $X$. Based on Corollary \ref{co2.3} and recalling that the sequence $%
\left\{ H_{n^{-1}}\left( B_{\mathcal{H}}\right) \right\} $ ($n$ ranging over
the positive integers) is increasing (see part (i) of Proposition \ref{pr2.1}%
), we can choose some integer $n\geq e$ such that $K\subset H_{n^{-1}}\left(
B_{\mathcal{H}}\right) $. This being so, let $\varepsilon \in E$, $%
\varepsilon \leq e$, and $u\in \Xi ^{p}\left( X,\mathcal{H},\lambda \right) $
be fixed in an arbitrary manner. Then%
\begin{equation*}
\int_{K}\left\vert u\left( H_{\varepsilon }\left( x\right) \right)
\right\vert ^{p}d\lambda \left( x\right) \leq \int_{H_{n^{-1}}\left( B_{%
\mathcal{H}}\right) }\left\vert u\left( H_{\varepsilon }\left( x\right)
\right) \right\vert ^{p}d\lambda \left( x\right) \text{.}
\end{equation*}%
In the sequel $\mathcal{\chi }_{A}$ denotes the characteristic function of a
set $A$ in $X$, and for convenience $f$ stands for the function in $%
L_{loc}^{1}\left( X\right) $ which is given by $f\left( x\right) =\left\vert
u\left( H_{n^{-1}\varepsilon }\left( x\right) \right) \right\vert ^{p}$\quad 
$\left( x\in X\right) $. Thus%
\begin{equation*}
\int_{H_{n^{-1}}\left( B_{\mathcal{H}}\right) }\left\vert u\left(
H_{\varepsilon }\left( x\right) \right) \right\vert ^{p}d\lambda \left(
x\right) =\int_{H_{n^{-1}}\left( B_{\mathcal{H}}\right) }f\left( H_{n}\left(
x\right) \right) d\lambda \left( x\right)
\end{equation*}%
\begin{equation*}
\qquad \qquad \qquad \qquad \qquad \qquad \qquad =\int f\left( H_{n}\left(
x\right) \right) \mathcal{\chi }_{H_{n^{-1}}\left( B_{\mathcal{H}}\right)
}\left( x\right) d\lambda \left( x\right)
\end{equation*}%
\begin{equation*}
\qquad \qquad \qquad \qquad \qquad \qquad =\int \left( f\mathcal{\chi }_{B_{%
\mathcal{H}}}\right) \left( H_{n}\left( x\right) \right) d\lambda \left(
x\right)
\end{equation*}%
\begin{equation*}
\qquad \qquad \qquad \qquad \qquad \qquad \qquad =h\left( n^{-1}\right) \int
f\left( x\right) \mathcal{\chi }_{B_{\mathcal{H}}}\left( x\right) d\lambda
\left( x\right)
\end{equation*}%
\begin{equation*}
\qquad \qquad \qquad \qquad \qquad \qquad =h\left( n^{-1}\right) \int_{B_{%
\mathcal{H}}}f\left( x\right) d\lambda \left( x\right)
\end{equation*}%
\begin{equation*}
\qquad \qquad \qquad \qquad \qquad \qquad \qquad \qquad =h\left(
n^{-1}\right) \int_{B_{\mathcal{H}}}\left\vert u\left( H_{n^{-1}\varepsilon
}\left( x\right) \right) \right\vert ^{p}d\lambda \left( x\right) \text{,}
\end{equation*}%
where we have made use of the equality $\mathcal{\chi }_{H_{n^{-1}}\left( B_{%
\mathcal{H}}\right) }=\mathcal{\chi }_{B_{\mathcal{H}}}\circ H_{n}$. Hence (%
\ref{eq4.2}) follows by letting $c=\left[ h\left( n^{-1}\right) \right] ^{%
\frac{1}{p}}$ and observing that $n^{-1}\varepsilon \leq e$.
\end{proof}

Now, for real $p\geq 1$, we define $\mathfrak{X}^{p}\left( X,\mathcal{H}%
,\lambda \right) $ to be the closure of $\Pi ^{\infty }\left( X,\mathcal{H}%
,\lambda \right) $ in $\Xi ^{p}\left( X,\mathcal{H},\lambda \right) $.
Provided with the $\Xi ^{p}\left( X,\mathcal{H},\lambda \right) $-norm, $%
\mathfrak{X}^{p}\left( X,\mathcal{H},\lambda \right) $ is a Banach space,
thanks to Lemma \ref{lem4.1}.

Having made this point, one can easily show the next result.

\begin{proposition}
\label{pr4.1} The mean value on $\left( X,\mathcal{H},\lambda \right) $
extends by continuity to a (unique) continuous linear form on $\mathfrak{X}%
^{p}\left( X,\mathcal{H},\lambda \right) $ still called $M$. Furthermore,
given $u\in \mathfrak{X}^{p}\left( X,\mathcal{H},\lambda \right) $ and a
fixed relatively compact open set $\Omega $ in $X$, we have $u\circ
H_{\varepsilon }\rightarrow M\left( u\right) $ in $L^{p}\left( \Omega
\right) $-weak as $\varepsilon \rightarrow \theta $, where $u\circ
H_{\varepsilon }$ is considered as defined on $\Omega $.
\end{proposition}

\begin{proof}
For $\psi \in \Pi ^{\infty }\left( X,\mathcal{H},\lambda \right) $, we have 
\begin{equation*}
\left\vert \int_{B_{\mathcal{H}}}\psi \left( H_{\varepsilon }\left( x\right)
\right) d\lambda \left( x\right) \right\vert \leq \lambda \left( B_{\mathcal{%
H}}\right) ^{\frac{1}{p^{\prime }}}\left\Vert \psi \right\Vert _{\Xi
^{p}}\qquad \left( \varepsilon \leq e\right) \text{,}
\end{equation*}%
where $p^{\prime }=\frac{p}{p-1}$. Letting $\varepsilon \rightarrow \theta $%
, this gives $\left\vert M\left( \psi \right) \right\vert \leq \lambda
\left( B_{\mathcal{H}}\right) ^{-\frac{1}{p}}\left\Vert \psi \right\Vert
_{\Xi ^{p}}$. Hence the first part of the proposition follows by extension
by continuity. The second part can be easily deduced from this by mere
routine (see, e.g., the proof of \cite{bib21}, Proposition \ref{pr2.5}).
\end{proof}

\begin{remark}
\emph{\label{rem4.1} We have} $L^{p}\left( X\right) \subset \mathfrak{X}%
^{p}\left( X,\mathcal{H},\lambda \right) $ \emph{with continuous embedding
and further }$M\left( u\right) =0$\emph{\ for }$u\in L^{p}\left( X\right) $%
\emph{.}
\end{remark}

We will later need to know about the behaviour of $u\circ H_{\varepsilon }$
(as $\varepsilon \rightarrow \theta $) for $u\in L^{p}\left( \Omega ;\Pi
^{\infty }\right) $, where $\Pi ^{\infty }=\Pi ^{\infty }\left( X,\mathcal{H}%
,\lambda \right) $ and $\Omega $ is an open set in $X$.

The character $\Omega $ throughout denotes a nonempty open set in $X$ (the
case $\Omega =X$ being included). In the case $\Omega \neq X$, $\Omega $ is
provided with the relative topology on $X$ and with the induced measure $%
\lambda _{\Omega }=\lambda \mathfrak{\mid }_{\Omega }$.

Now, let $\varepsilon \in E$ be fixed. For $u\in L_{loc}^{1}\left( \Omega
\times X\right) $, we put 
\begin{equation}
u^{\varepsilon }\left( x\right) =u\left( x,H_{\varepsilon }\left( x\right)
\right) \qquad \left( x\in \Omega \right)  \label{eq4.3}
\end{equation}%
whenever the right-hand side makes sense. On this point it is worth noting
that the set $D\left( \varepsilon \right) =\left\{ \left( x,H_{\varepsilon
}\left( x\right) \right) :x\in \Omega \right\} \subset X\times X$ is
negligible for the product measure $\lambda \otimes \lambda $ (see \cite%
{bib18}, Section 4), so that, except in obvious cases such as that in which $%
u$ is a continuous function on $\Omega \times X$, it is in general a
delicate matter to give a meaning to the trace $u\mathfrak{\mid }_{D\left(
\varepsilon \right) }$ of $u\in L_{loc}^{1}\left( \Omega \times X\right) $.
The least obvious case that concerns us is when $u$ lies in $L^{p}\left(
\Omega ;\Pi ^{\infty }\right) $.

More generally, let $A$ be a closed vector subspace of $\mathcal{B}\left(
X\right) $, $A$ being equipped with the supremum norm. Let $1\leq p\leq
+\infty $ be arbitrarily fixed. For $u\in L^{p}\left( \Omega ;A\right) $,
let $\mathcal{N}\subset \Omega $ be a $\lambda $-negligible set such that
for each fixed $x\in \Omega \backslash \mathcal{N}$, the map $y\rightarrow
u\left( x,y\right) $ of $X$ into $\mathbb{C}$ lies in $A$. We may then
consider the mapping $z\rightarrow u\left( x,H_{\varepsilon }\left( z\right)
\right) $ which sends continuously $X$ into $\mathbb{C}$. Hence the complex
number $u\left( x,H_{\varepsilon }\left( x\right) \right) $ is well defined
for any $x\in \Omega \backslash \mathcal{N}$. This yields a function $%
x\rightarrow u^{\varepsilon }\left( x\right) $ from $\Omega $ to $\mathbb{C}$
defined (almost everywhere in $\Omega $) by (\ref{eq4.3}). Let us make this
more precise.

\begin{lemma}
\label{lem4.2} Let $1\leq p\leq +\infty $. For each $u\in L^{p}\left( \Omega
;A\right) $, we have $u^{\varepsilon }\in L^{p}\left( \Omega \right) $ and
further the transformation $u\rightarrow u^{\varepsilon }$ maps linearly and
continuously $L^{p}\left( \Omega ;A\right) $ into $L^{p}\left( \Omega
\right) $ with 
\begin{equation}
\left\Vert u^{\varepsilon }\right\Vert _{L^{p}\left( \Omega \right) }\leq
\left\Vert u\right\Vert _{L^{p}\left( \Omega ;A\right) }\qquad \qquad \left(
u\in L^{p}\left( \Omega ;A\right) \right) \text{.}  \label{eq4.4}
\end{equation}
\end{lemma}

\begin{proof}
Let $u\in L^{p}\left( \Omega ;A\right) $. The first point is to check that $%
u^{\varepsilon }$ is measurable (for $\lambda $). For this purpose fix
freely some compact set $K\subset \Omega $ and some real $\eta >0$. Since $u$
is measurable from $\Omega $ to $A$, there is some compact set $K^{\prime
}\subset K$ such that $\lambda \left( K\backslash K^{\prime }\right) \leq
\eta $ and further $u\mathfrak{\mid }_{K^{\prime }}$ (the restriction of $u$
to $K^{\prime }$) maps continuously $K^{\prime }$ into $A$. Clearly the
measurability of $u^{\varepsilon }$ will have been established if we can
verify that $u^{\varepsilon }$ is continuous on $K^{\prime }$. To show this,
let $a\in K^{\prime }$. Fix freely $\alpha >0$. Considering that the
function $x\rightarrow u\left( a,H_{\varepsilon }\left( x\right) \right) $
is continuous on $\Omega $ and in particular at $x=a$, we can choose a
neighbourhood $U$ of $a$ in $\Omega $ such that $\left\vert u\left(
a,H_{\varepsilon }\left( x\right) \right) -u\left( a,H_{\varepsilon }\left(
a\right) \right) \right\vert \leq \frac{\alpha }{2}$ for all $x\in U$. On
the other hand, it is clear that there exists a neighbourhood $V$ of $a$ in $%
\Omega $ such that $\left\Vert u\left( x,.\right) -u\left( a,.\right)
\right\Vert _{\infty }\leq \frac{\alpha }{2}$ for all $x\in V\cap K^{\prime
} $. Hence $\left\vert u^{\varepsilon }\left( x\right) -u^{\varepsilon
}\left( a\right) \right\vert \leq \alpha $ for all $x\in U\cap V\cap
K^{\prime }$, and the measurability of $u^{\varepsilon }$\ follows thereby.
Finally, we evidently have $\left\vert u\left( x,y\right) \right\vert \leq
\left\Vert u\left( x,.\right) \right\Vert _{\infty }$ for all $x\in \Omega
\backslash \mathcal{N}$ and all $y\in X$. It follows that $\left\vert
u^{\varepsilon }\left( x\right) \right\vert \leq \left\Vert u\left(
x,.\right) \right\Vert _{\infty }$ for all $x\in \Omega \backslash \mathcal{N%
}$. Recalling that $u^{\varepsilon }$ is measurable, we deduce that $%
u^{\varepsilon }$ lies in $L^{p}\left( \Omega \right) $ and further (\ref%
{eq4.4}) holds. The linearity of the transformation $u\rightarrow
u^{\varepsilon }$ being evident, the proof is complete.
\end{proof}

We are now in a position to investigate the desired asymptotic properties.
In the sequel $A$ denotes a closed vector subspace of $\Pi ^{\infty }=\Pi
^{\infty }\left( X,\mathcal{H},\lambda \right) $ (the case $A=\Pi ^{\infty }$
included), $A$ being provided with the supremum norm. It is worth recalling
that $\mathcal{K}\left( \Omega ;A\right) $ denotes the space of continuous
functions of $\Omega $ into $A$ with compact supports. For $u\in L^{p}\left(
\Omega ;A\right) $, $1\leq p\leq +\infty $, we put $\widetilde{u}\left(
x\right) =M\left( u\left( x\right) \right) $, $x\in \Omega $, which gives a
function $\widetilde{u}\in L^{p}\left( \Omega \right) $. If in particular $%
u\in \mathcal{K}\left( \Omega ;A\right) $, then $\widetilde{u}\in \mathcal{K}%
\left( \Omega \right) $.

\begin{proposition}
\label{pr4.2} The following assertions are true for any closed vector
subspace $A$ of $\Pi ^{\infty }\left( X,\mathcal{H},\lambda \right) $:

(i) For $u\in \mathcal{K}\left( \Omega ;A\right) $, we have $u^{\varepsilon
}\rightarrow \widetilde{u}$ in $L^{\infty }\left( \Omega \right) $-weak $%
\ast $ as $\varepsilon \rightarrow \theta $.

(ii) For $u\in L^{p}\left( \Omega ;A\right) $ $\left( 1\leq p<+\infty
\right) $, we have $u^{\varepsilon }\rightarrow \widetilde{u}$ in $%
L^{p}\left( \Omega \right) $-weak as $\varepsilon \rightarrow \theta $.
\end{proposition}

\begin{proof}
We first prove (i). Let $\mathcal{K}\left( \Omega \right) \otimes A$ be the
space of all complex functions $\psi $ on $\Omega \times X$ of the form 
\begin{equation*}
\psi \left( x,y\right) =\sum \varphi _{i}\left( x\right) w_{i}\left(
y\right) \qquad \left( x\in \Omega ,\text{ }y\in X\right)
\end{equation*}%
with a summation of finitely many terms, $\varphi _{i}\in \mathcal{K}\left(
\Omega \right) $, $w_{i}\in A$. According to the definition of $\Pi ^{\infty
}\left( X,\mathcal{H},\lambda \right) $, it is clear that (i) holds true
with $\mathcal{K}\left( \Omega \right) \otimes A$ in place of $\mathcal{K}%
\left( \Omega ;A\right) $. But this implies the whole of (i) in an obvious
way because $\mathcal{K}\left( \Omega \right) \otimes A$ is dense in $%
\mathcal{K}\left( \Omega ;A\right) $ (provided with the inductive limit
topology) (see, e.g., \cite{bib5}, p.46). It remains to check (ii). But this
follows by (i) and use of the density of $\mathcal{K}\left( \Omega ;A\right) 
$ in $L^{p}\left( \Omega ;A\right) $ (the way of proceeding is a routine
exercise left to the reader).
\end{proof}

\section{ABSORPTIVE CONTINUOUS $\mathbb{R}$-GROUP ACTIONS ON\ LOCALLY
COMPACT ABELIAN GROUPS}

In all that follows, $G$ denotes a locally compact abelian group additively
written, not reduced to its neutral element.

\subsection{\textbf{Basic results}}

We start with one simple but fundamental result.

\begin{proposition}
\label{pr5.1} Let $E$ be an $\mathbb{R}$-group, and let $\mathcal{H=}\left(
H_{\varepsilon }\right) _{\varepsilon \in E}$ be an absorptive continuous
action of $E$ on $G$. We suppose moreover that $\mathcal{H}$ is compatible
with the group operation in $G$, i.e., $H_{\varepsilon }\left( x+y\right)
=H_{\varepsilon }\left( x\right) +H_{\varepsilon }\left( y\right) $ for $x$, 
$y\in G$, $\varepsilon \in E$. Then:

(i) The center of $\mathcal{H}$ coincides with the neutral element of $G$.

(ii) $G$ is metrizable.

(iii) $G$ is $\sigma $-compact, noncompact and nondiscrete.

(iv) Haar measure on $G$ is nontrivial and homogeneous (for $\mathcal{H}$).
\end{proposition}

\begin{proof}
The neutral element of $G$ is clearly an invariant point (for $\mathcal{H}$%
). But there is only one invariant point which is the center of $\mathcal{H}$
(see Proposition \ref{pr2.2}). This shows (i). By combining this with
Corollary \ref{co2.4}, we see that the neutral element of $G$ has a
countable base of neighbourhoods, which implies (ii) (see, e.g., \cite{bib6}%
, chap.IX, p.23); (iii) follows by Corollary \ref{co2.1}. Finally, let $%
\lambda $ be Haar measure on $G$. Evidently $\lambda $ is nonzero. On the
other hand, because $G$ is nondiscrete, we have $\lambda \left( \left\{
\omega \right\} \right) =0$ (see, e.g., \cite{bib8}, p.242), where $\omega $
is the neutral element of $G$ (coinciding with the center of $\mathcal{H}$).
Thus $\lambda \neq \delta _{\omega }$, hence $\lambda $ is nontrivial. Now,
for fixed $\varepsilon \in E$, it is clear that $H_{\varepsilon }\left(
\lambda \right) $ is a Haar measure on $G$, since $H_{\varepsilon }$ is a
topological (group) automorphism of $G$. Hence, by unicity (up to a
multiplicative constant) of Haar measure, there is some $c\left( \varepsilon
\right) >0$ such that $H_{\varepsilon }\left( \lambda \right) =c\left(
\varepsilon \right) \lambda $, which shows (iv).
\end{proof}

Let $E$ be an $\mathbb{R}$-group, and let $\mathcal{H=}\left( H_{\varepsilon
}\right) _{\varepsilon \in E}$ be an absorptive continuous action of $E$ on $%
G$, which we moreover assume to be compatible with the group operation in $G$%
. As usual, $\widehat{G}$ denotes the dual group of $G$, that is, the set of
all continuous homomorphisms of $G$ into the unit circle $\mathbb{U}=\left\{
z\in \mathbb{C}:\left\vert z\right\vert =1\right\} $. Points in $\widehat{G}$
are referred to as the continuous characters of $G$. Provided with
multiplication (the product of two continuous characters is defined in an
obvious way) and with the topology of compact convergence, $\widehat{G}$ is
a locally compact abelian group. The value of $\mathcal{\gamma }\in \widehat{%
G}$ at some point $x\in G$ will be denoted by $\left\langle \mathcal{\gamma }%
,x\right\rangle $. Recalling that $\omega $ denotes the center of $\mathcal{H%
}$ which is also the neutral element of $G$, we will use the symbol $%
\widehat{\omega }$ to denote the identity of $\widehat{G}$, i.e., $\widehat{%
\omega }$ is the continuous character of $G$ defined by $\left\langle 
\widehat{\omega },x\right\rangle =1$ for $x\in G$. We also recall that the
dual of $H_{\varepsilon }$ is defined to be the map $\widehat{H}%
_{\varepsilon }:\widehat{G}\rightarrow \widehat{G}$ such that $\langle 
\widehat{H}_{\varepsilon }\left( \mathcal{\gamma }\right) ,x\rangle
=\left\langle \mathcal{\gamma },H_{\varepsilon }\left( x\right)
\right\rangle $\quad $(\mathcal{\gamma }\in \widehat{G}$, $x\in G)$. It is
immediate that $\widehat{H}_{\varepsilon }$ is a topological automorphism of 
$\widehat{G}$, and further the family $\widehat{\mathcal{H}}=(\widehat{H}%
_{\varepsilon })_{\varepsilon \in E}$ is an action of $E$ on $\widehat{G}$
which is compatible with the group operation in $\widehat{G}$. We call $%
\widehat{\mathcal{H}}$ the dual of the action $\mathcal{H}$. Our purpose is
precisely to establish that the dual action $\widehat{\mathcal{H}}$ is
continuous and absorptive. Let us observe beforehand that $\widehat{G}\neq
\left\{ \widehat{\omega }\right\} $, since $G\neq \left\{ \omega \right\} $
(see \cite{bib14}, p. 38, Corollaire 2.3).

\begin{proposition}
\label{pr5.2} The dual action $\widehat{\mathcal{H}}=(\widehat{H}%
_{\varepsilon })_{\varepsilon \in E}$ is continuous and absorptive.
\end{proposition}

\begin{proof}
Let us show the continuity. Specifically, the aim is to establish that the
map $\left( \varepsilon ,\mathcal{\gamma }\right) \rightarrow \widehat{H}%
_{\varepsilon }\left( \mathcal{\gamma }\right) $ sends continuously $E\times 
\widehat{G}$ into $\widehat{G}$, or equivalently that $\left( \varepsilon ,%
\mathcal{\gamma }\right) \rightarrow \widehat{H}_{\varepsilon ^{-1}}\left( 
\mathcal{\gamma }\right) $ maps continuously $E\times \widehat{G}$ into $%
\widehat{G}$. To do this, fix freely $\left( \varepsilon _{0},\mathcal{%
\gamma }_{0}\right) \in E\times \widehat{G}$. Let $V$ be a neighbourhood of $%
\widehat{H}_{\varepsilon _{0}^{-1}}\left( \mathcal{\gamma }_{0}\right) $ in $%
\widehat{G}$. We may assume without loss of generality that%
\begin{equation*}
V=\left\{ \mathcal{\gamma }\in \widehat{G}:\left\vert \left\langle \mathcal{%
\gamma },x\right\rangle -\langle \widehat{H}_{\varepsilon _{0}^{-1}}\left( 
\mathcal{\gamma }_{0}\right) ,x\rangle \right\vert \leq \alpha \text{ for
all }x\in K\right\} \text{,}
\end{equation*}%
where $K$ is a compact set in $G$ and $\alpha $ is a positive real number.
Now, recalling that each continuous character of $G$ is uniformly
continuous, we see that we may consider a neighbourhood $W$ of $\omega $ in $%
G$ such that $\left\vert \left\langle \mathcal{\gamma }_{0},x\right\rangle
-\left\langle \mathcal{\gamma }_{0},y\right\rangle \right\vert \leq \frac{%
\alpha }{2}$ for all $x$, $y\in G$ with $x-y\in W$. On the other hand, let $%
I $ be a compact neighbourhood of $\varepsilon _{0}$ in $E$. Observe that
the map $\left( \varepsilon ,x\right) \rightarrow H_{\varepsilon
^{-1}}\left( x\right) $ is uniformly continuous of $I\times K$ (a compact
set) into $G$; thus, as $x$ ranges over $K$, the maps $\varepsilon
\rightarrow H_{\varepsilon ^{-1}}\left( x\right) $ of $I$ into $G$ form an
equicontinuous family at $\varepsilon _{0}$ (see \cite{bib6}, Chap.X, p.13,
Proposition 2). This yields a neighbourhood $J$ of $\varepsilon _{0}$ in $E$
such that $H_{\varepsilon ^{-1}}\left( x\right) -H_{\varepsilon
_{0}^{-1}}\left( x\right) \in W$ for all $\varepsilon \in J$ and all $x\in K$%
. Therefore%
\begin{equation*}
\left\vert \left\langle \mathcal{\gamma }_{0},H_{\varepsilon ^{-1}}\left(
x\right) \right\rangle -\langle \mathcal{\gamma }_{0},H_{\varepsilon
_{0}^{-1}}\left( x\right) \rangle \right\vert \leq \frac{\alpha }{2}\qquad
\left( \varepsilon \in J\text{, }x\in K\right) \text{.}
\end{equation*}%
With this in mind, we next consider a compact neighbourhood $F$ of $\omega $
in $G$ and some $\eta \mathfrak{\in }E$ with $\eta >\varepsilon _{0}$, such
that $H_{\varepsilon ^{-1}}\left( K\right) \subset F$ for $\varepsilon \leq
\eta $ (let $F_{0}$ be a compact neighbourhood of $\omega $ in $G$, and
according to part (ii) of Proposition \ref{pr2.2}, let $\eta _{0}\in E$ be
such that $H_{\varepsilon ^{-1}}\left( K\right) \subset F_{0}$ for $%
\varepsilon \leq \eta _{0}$. Choose a positive integer $n$ such that $n\eta
_{0}>\varepsilon _{0}$. Then $\eta =n\eta _{0}$ and $F=H_{n^{-1}}\left(
F_{0}\right) $ fit the statement). Finally, we put $J_{0}=J\cap \left(
\theta ,\eta \right) $ (a neighbourhood of $\varepsilon _{0}$ in $E$) and $%
U_{0}=\left\{ \mathcal{\gamma }\in \widehat{G}:\left\vert \left\langle 
\mathcal{\gamma },x\right\rangle -\left\langle \mathcal{\gamma }%
_{0},x\right\rangle \right\vert \leq \frac{\alpha }{2}\text{ for }x\in
F\right\} $ (a neighbourhood of $\mathcal{\gamma }_{0}$ in $\widehat{G}$).
Then $\widehat{H}_{\varepsilon ^{-1}}\left( U_{0}\right) \subset V$ for $%
\varepsilon \in J_{0}$. In other words, we have $\widehat{H}_{\varepsilon
^{-1}}\left( \mathcal{\gamma }\right) \in V$ for $\left( \varepsilon ,%
\mathcal{\gamma }\right) \in J_{0}\times U_{0}$. The continuity of $\widehat{%
\mathcal{H}}$ follows.

\noindent We now check the absorptiveness. Clearly it is enough to verify
that any compact set in $\widehat{G}$ is $\widehat{\mathcal{H}}$-absorbed by
any neighbourhood of $\widehat{\omega }$. To this end let $\Gamma $ be a
compact set in $\widehat{G}$ and $V$ be a neighbourhood of $\widehat{\omega }
$ in $\widehat{G}$. For an obvious reason we may assume without loss of
generality that $V=\left\{ \mathcal{\gamma }\in \widehat{G}:\left\vert
\left\langle \mathcal{\gamma },x\right\rangle -1\right\vert \leq \alpha 
\text{ for }x\in K\right\} $, where $K$ is some compact set in $G$ and $%
\alpha $ some positive real. By Ascoli's theorem, $\Gamma $ is an
equicontinuous family of complex maps on $G$ and hence there is a
neighbourhood $U$ of $\omega $ in $G$ such that $\left\vert \left\langle 
\mathcal{\gamma },y\right\rangle -1\right\vert \leq \alpha $ for $y\in U$
and $\mathcal{\gamma }\in \Gamma $. Finally, choose $\eta \in E$ such that $%
H_{\varepsilon ^{-1}}\left( K\right) \subset U$ for $\varepsilon \leq \eta $
(part (ii) of Proposition \ref{pr2.2}). Then, from all that one quickly
deduces that $\widehat{H}_{\varepsilon ^{-1}}\left( \Gamma \right) \subset V$
for $\varepsilon \leq \eta $. This completes the proof.
\end{proof}

In the sequel $\lambda $ denotes Haar measure on $G$.

\begin{proposition}
\label{pr5.3} Let the notation and hypotheses be as above. Then the mean
value $M$ on $\left( G,\mathcal{H},\lambda \right) $ is translation
invariant. Specifically, for $u\in \Pi ^{\infty }\left( G,\mathcal{H}%
,\lambda \right) $ and $a\in G$, we have $\tau _{a}u\in \Pi ^{\infty }\left(
G,\mathcal{H},\lambda \right) $ and further $M\left( \tau _{a}u\right)
=M\left( u\right) $, where $\tau _{a}u\left( x\right) =u\left( x-a\right) $%
\quad $\left( x\in G\right) $.
\end{proposition}

\begin{proof}
Clearly we are through if we can show that\ as $\varepsilon \rightarrow
\theta $,%
\begin{equation*}
\int \varphi \left( x\right) \left( \tau _{a}u\right) \left( H_{\varepsilon
}\left( x\right) \right) d\lambda \left( x\right) \rightarrow M\left(
u\right) \int \varphi \left( x\right) d\lambda \left( x\right)
\end{equation*}%
for any $\varphi \in \mathcal{K}\left( G\right) $. For this purpose, fix one
such $\varphi $. Let $B$ be a compact neighbourhood\ of $\omega $, and let $%
c>0$\ be a constant such that $\left\Vert u\right\Vert _{\infty }\lambda
\left( Supp\varphi -B\right) \leq c$. Now, let $\eta >0$. By uniform
continuity, consider a neighbourhood $V$ of $\omega $ such that $\left\vert
\varphi \left( x\right) -\varphi \left( y\right) \right\vert \leq \frac{\eta 
}{c}$ whenever $x$, $y\in G$ with $x-y\in V$. Furthermore, fix $\alpha \in E$
such that $H_{\varepsilon ^{-1}}\left( a\right) \in V\cap \left( -B\right) $
whenever $\varepsilon \leq \alpha $. Then%
\begin{equation*}
\left\vert \int \varphi \left( x\right) u\left( H_{\varepsilon }\left(
x\right) -a\right) d\lambda \left( x\right) -\int \varphi \left( x\right)
u\left( H_{\varepsilon }\left( x\right) \right) d\lambda \left( x\right)
\right\vert
\end{equation*}%
\begin{equation*}
=\left\vert \int u\left( H_{\varepsilon }\left( y\right) \right) \left[
\varphi \left( y+H_{\varepsilon ^{-1}}\left( a\right) \right) -\varphi
\left( y\right) \right] d\lambda \left( y\right) \right\vert \leq \eta
\end{equation*}%
provided $\varepsilon \leq \alpha $. Hence the proposition follows.
\end{proof}

Thus, the mean value on $\left( G,\mathcal{H},\lambda \right) $ is a mean
value on the locally compact group $G$ (see \cite{bib20}, Definition \ref%
{def2.1}). On the other hand, the mean value on $\left( G,\mathcal{H}%
,\lambda \right) $ generalizes the mean value for ponderable functions on $%
\mathbb{R}^{N}$, and there is no serious difficulty in showing that the
results established in subsection 4.1 of \cite{bib20} carry over mutatis
mutandis to the present setting. In particular there is the following

\begin{proposition}
\label{pr5.4} Let $f\in L^{1}\left( G\right) $. If $u\in \Pi ^{\infty
}\left( G,\mathcal{H},\lambda \right) $, then $f\ast u\in \Pi ^{\infty
}\left( G,\mathcal{H},\lambda \right) $ and 
\begin{equation*}
M\left( f\ast u\right) =M\left( u\right) \int f\left( y\right) d\lambda
\left( y\right) \text{.}
\end{equation*}
\end{proposition}

\begin{proof}
This follows by a simple adaptation of Proposition \ref{pr4.1} of \cite%
{bib20}.
\end{proof}

\subsection{\textbf{Almost periodic functions}}

Let the notation and hypotheses be as above. Our goal here is to show that
the mean value on $\left( G,\mathcal{H},\lambda \right) $ is an extension of
the mean value for the almost periodic continuous complex functions on $G$.
For the benefit of the reader we recall that a map $u:G\rightarrow \mathbb{C}
$ is called a \textit{Bohr almost periodic function, }or an almost periodic
continuous complex function on $G$, if $u$ lies in $\mathcal{B}\left(
G\right) $ and further if the translates $\tau _{a}u$ $\left( a\in G\right) $
form a relatively compact set in $\mathcal{B}\left( G\right) $. The space of
such functions is commonly denoted by $AP\left( G\right) $, and is a Banach
space under the supremum norm. More precisely, $AP\left( G\right) $ with the
supremum norm and the usual algebra operations in $\mathcal{B}\left(
G\right) $ is a commutative $\mathcal{C}^{\ast }$-algebra with identity. On
the other hand, given $u\in AP\left( G\right) $, it can be shown that the
closed convex hull of the set $\left\{ \tau _{a}u:a\in G\right\} $ in $%
\mathcal{B}\left( G\right) $ contains one and only one constant $m\left(
u\right) $ called the mean of $u$ (see, e.g., \cite{bib14}, p.94, and \cite%
{bib20}). This yields a map $u\rightarrow m\left( u\right) $ of $AP\left(
G\right) $ into $\mathbb{C}$, called the mean value on $G$ for $AP\left(
G\right) $ (see \cite{bib20}).

\begin{proposition}
\label{pr5.5} Let the notation and hypotheses be as before. Then, the mean
value $M$ on $\left( G,\mathcal{H},\lambda \right) $ is an extension of the
mean value $m$ for $AP\left( G\right) $. Specifically, we have $AP\left(
G\right) \subset \Pi ^{\infty }\left( G,\mathcal{H},\lambda \right) $ and $%
M\left( u\right) =m\left( u\right) $ for $u\in AP\left( G\right) $.
\end{proposition}

\begin{proof}
The whole problem reduces to showing that for $u\in AP\left( G\right) $,%
\begin{equation}
u\circ H_{\varepsilon }\rightarrow m\left( u\right) \text{ in }L^{\infty
}\left( G\right) \text{-weak }\ast \text{ as }\varepsilon \rightarrow \theta 
\text{.}  \label{eq5.1}
\end{equation}%
Let $\mathcal{P}$ be the set of all finite linear combinations of continuous
characters of $G$. Recalling that the continuous characters of $G$ lie in $%
AP\left( G\right) $ and further $AP\left( G\right) $ coincides with the
closure of $\mathcal{P}$ in $\mathcal{B}\left( G\right) $ (see, e.g., \cite%
{bib14}, Chap.5, and \cite{bib16}, Chap.10), we see that it is enough to
show (\ref{eq5.1}) for $u\in \mathcal{P}$ or better still, for $u=\mathcal{%
\gamma }\in \widehat{G}$. In other words, the proposition is established if
we can verify that, given any arbitrary $f\in L^{1}\left( G\right) $ ($f$
independent of $\varepsilon $, of course), we have as $\varepsilon
\rightarrow \theta $,%
\begin{equation*}
\int \left\langle \mathcal{\gamma },H_{\varepsilon }\left( x\right)
\right\rangle f\left( x\right) d\lambda \left( x\right) \rightarrow m\left( 
\mathcal{\gamma }\right) \int f\left( x\right) d\lambda \left( x\right)
\end{equation*}%
or equivalently%
\begin{equation*}
\int \overline{\left\langle \mathcal{\gamma },H_{\varepsilon }\left(
x\right) \right\rangle }f\left( x\right) d\lambda \left( x\right)
\rightarrow m\left( \overline{\mathcal{\gamma }}\right) \int f\left(
x\right) d\lambda \left( x\right)
\end{equation*}%
for all $\mathcal{\gamma }\in \widehat{G}$, where $\overline{\mathcal{\gamma 
}}$ denotes the map $x\rightarrow \overline{\left\langle \mathcal{\gamma }%
,x\right\rangle }$ (complex conjugate of $\left\langle \mathcal{\gamma }%
,x\right\rangle $) of $G$ into $\mathbb{C}$. This convergence result is
trivial if $\mathcal{\gamma }=\widehat{\omega }$ (the identity of $\widehat{G%
}$), since $m\left( \widehat{\omega }\right) =1$. So we assume in the sequel
that $\mathcal{\gamma }\neq \widehat{\omega }$. Recalling that we have in
this case $m\left( \overline{\mathcal{\gamma }}\right) =0$ (this is
well-known), we see that we are through if we can check that $%
\lim_{\varepsilon \rightarrow \theta }\mathcal{F}f\left( \widehat{H}%
_{\varepsilon }\left( \mathcal{\gamma }\right) \right) =0$, where $\mathcal{F%
}f$ denotes the Fourier transform of $f$. But $\lim_{\varepsilon \rightarrow
\theta }\widehat{H}_{\varepsilon }\left( \mathcal{\gamma }\right) =\infty $
(the point at infinity of the Alexandroff compactification of $\widehat{G}$%
), as is straightforward by combining Proposition \ref{pr5.2} with Corollary %
\ref{co2.2}. Hence the desired result follows by the Riemann-Lebesgue lemma.
\end{proof}

Thus, Proposition \ref{pr4.2} is valid with $X=G$, $\mathcal{H}$ compatible
with the group operation in $G$, $\lambda $ Haar measure on $G$, $A=AP\left(
G\right) $.

\section{HOMOGENIZATION\ ALGEBRAS}

The aim here is to point out a close connection between homogenization
theory and absorptive continuous $\mathbb{R}$-group actions.

Let $\left( X,\mathcal{H},\lambda \right) $ be a homogenizer with $\mathcal{%
H=}\left( H_{\varepsilon }\right) _{\varepsilon \in E}$, where $E$ is the
acting $\mathbb{R}$-group. The basic notation is as before. In particular $M$
is the mean value on $\left( X,\mathcal{H},\lambda \right) $, and the
entities $e$, $\theta $ attached to $E$ are defined in subsection 2.1.

\begin{definition}
\emph{\label{def6.1} We call a }homogenization algebra\emph{\ (or an
H-algebra) on }$\left( X,\mathcal{H},\lambda \right) $\emph{, any closed
subalgebra }$A$\emph{\ of }$\mathcal{B}\left( X\right) $\emph{\ with the
following properties:}

\emph{(HA)}$_{1}$\emph{\ }$A$\emph{\ with the supremum norm is separable.}

\emph{(HA)}$_{2}$\emph{\ }$A$\emph{\ contains the constants.}

\emph{(HA)}$_{3}$\emph{\ If }$u\in A$\emph{, then }$\overline{u}$\emph{\ (}$%
\overline{u}$\emph{\ the complex conjugate of }$u$\emph{)}

\emph{(HA)}$_{4}$\emph{\ }$A\subset D\left( M\right) =\Pi ^{\infty }\left( X,%
\mathcal{H},\lambda \right) $\emph{.}
\end{definition}

In the sequel the $H$-algebra $A$ is assumed to be equipped with the
supremum norm. Thus, $A$ is a commutative $\mathcal{C}^{\ast }$-algebra with
identity. We denote the spectrum of $A$ by $\Delta \left( A\right) $ (i.e., $%
\Delta \left( A\right) $ is the set of all nonzero multiplicative linear
forms on $A$), and we endow $\Delta \left( A\right) $ with\ the relative
weak $\ast $ topology on $A^{\prime }$ (topological dual of $A$). As is
classical, $\Delta \left( A\right) $ is a compact space (see, e.g., \cite%
{bib8}, p.304, and \cite{bib16}, p.71). Furthermore, in view of (HA)$_{1}$, $%
\Delta \left( A\right) $ is metrisable. The Gelfand transformation on $A$
will be denoted by $\mathcal{G}$. It is worth recalling that $\mathcal{G}$
is the mapping of $A$ into $\mathcal{C}\left( \Delta \left( A\right) \right) 
$ such that $\mathcal{G}\left( u\right) \left( s\right) =\left\langle
s,u\right\rangle $\quad $\left( s\in \Delta \left( A\right) \text{, }u\in
A\right) $, where the brackets denote the duality pairing between $A^{\prime
}$ and $A$. It is also worth knowing that $\mathcal{G}$ is an isometric
isomorphism of the $\mathcal{C}^{\ast }$-algebra $A$ onto the $\mathcal{C}%
^{\ast }$-algebra $\mathcal{C}\left( \Delta \left( A\right) \right) $ (see, 
\cite{bib16}, p.277).

The appropriate measure on $\Delta \left( A\right) $ is the so-called $M$%
-measure for $A$ denoted below by $\beta $.

\begin{proposition}
\label{pr6.1} There exists a unique Radon measure $\beta $ on $\Delta \left(
A\right) $ such that $M\left( u\right) =\int_{\Delta \left( A\right) }%
\mathcal{G}\left( u\right) \left( s\right) d\beta \left( s\right) $ for all $%
u\in A$. Furthermore, $\beta $ is positive and of total mass $1$.
\end{proposition}

\begin{proof}
The functional $\varphi \rightarrow M\left( \mathcal{G}^{-1}\left( \varphi
\right) \right) \quad \left( \varphi \in \mathcal{C}\left( \Delta \left(
A\right) \right) \right) $ is a positive continuous linear form on $\mathcal{%
C}\left( \Delta \left( A\right) \right) $ attaining the value $1$ on the
constant function $1$. Therefore the proposition follows.
\end{proof}

In the sequel it is always assumed that $\Delta \left( A\right) $ is
provided with the measure $\beta $.

We present below a few examples of $H$-algebras.

\begin{example}
\emph{\label{ex6.1} We assume that }$X$\emph{\ is metrizable. Let} $\mathcal{%
B}_{\infty }\left( X\right) $ \emph{be as in Example \ref{ex4.1}. We claim
that this is a homogenization algebra on }$\left( X,\mathcal{H},\lambda
\right) $\emph{. Indeed,} \emph{letting }$A=\mathcal{B}_{\infty }\left(
X\right) $, \emph{it is clear that we only need to check (HA)}$_{1}$\emph{.
Because }$X$\emph{\ is both metrizable and }$\sigma $\emph{-compact
(Corollary \ref{co2.1}), its topology has a countable basis} \emph{and hence}
$\mathcal{B}_{\infty }\left( X\right) $ \emph{under the supremum norm is
separable (see \cite{bib6}, Chap.IX, pp.18, 21, and Chap.X, p.25).}
\end{example}

\begin{example}
\emph{\label{ex6.2}\ Let the setting be as in Example 4.2. Then, the space} $%
A=\mathcal{C}_{per}\left( Y\right) $ \emph{is an H-algebra on} $\left( 
\mathbb{R}^{N},\mathcal{H},\lambda \right) $.
\end{example}

\begin{example}
\emph{\label{ex6.3}} \emph{Let }$\left( G,\mathcal{H},\lambda \right) $\emph{%
\ be a homogenizer in which: }$G$\emph{\ is a locally compact abelian group
not reduced to its neutral element, the action} $\mathcal{H=}\left(
H_{\varepsilon }\right) _{\varepsilon \in E}$ \emph{is compatible with the
group operation, and }$\lambda $ \emph{is Haar measure on }$G$\emph{\ (see
Section 5). The space }$AP\left( G\right) $\emph{\ introduced in subsection
5.2 is a closed subalgebra of} $\mathcal{B}\left( G\right) $ \emph{%
satisfying\ conditions (HA)}$_{2}$\emph{-(HA)}$_{4}$\emph{\ (see Proposition %
\ref{pr5.4}). Unfortunately }$AP\left( G\right) $\emph{\ fails to carry out
(HA)}$_{1}$\emph{\ and hence we are led to restrict ourselves to some
specific subalgebras.} \emph{Let} $\mathcal{R}$ \emph{be a countable
subgroup of }$\widehat{G}$\emph{. We define }$AP_{\mathcal{R}}\left(
G\right) =\left\{ u\in AP\left( G\right) :Sp\left( u\right) \subset \mathcal{%
R}\right\} $\emph{\ with} $Sp\left( u\right) =\left\{ \mathcal{\gamma }\in 
\widehat{G}:m\left( \overline{\mathcal{\gamma }}u\right) \neq 0\right\} $ 
\emph{(spectrum of }$u$\emph{), where }$m$\emph{\ denotes the mean value on }%
$G$\emph{\ for} $AP\left( G\right) $.\emph{\ Note that the definition of }$%
AP_{\mathcal{R}}\left( G\right) $\emph{\ makes sense because the spectrum of
any function in }$AP\left( G\right) $\emph{\ turns out\ to be a countable
set. This being so, let }$\mathcal{P}_{\mathcal{R}}$ \emph{be the linear span%
}\ \emph{of} $\mathcal{R}$, \emph{that is, the space of all finite linear
combinations of elements of} $\mathcal{R}$. $AP_{\mathcal{R}}\left( G\right) 
$ \emph{coincides with the closure of }$\mathcal{P}_{\mathcal{R}}$ \emph{in} 
$\mathcal{B}\left( G\right) $ \emph{(see \cite{bib14}, p.93, Proposition
5.4) and} \emph{therefore} $AP_{\mathcal{R}}\left( G\right) $ \emph{is
separable under the supremum norm. Thus, it becomes an} \emph{elementary
exercise to check that} $AP_{\mathcal{R}}\left( G\right) $ \emph{is an }$H$%
\emph{-algebra on }$\left( G,\mathcal{H},\lambda \right) $\emph{.}
\end{example}

We will conclude the present section by discussing sigma-convergence.

Let $\Omega $ be a nonempty open set\ in $X$, with the same notation and
assumptions as in Section 4. On the other hand, let $S$ be a family in $E$
admitting $\theta $ as an accumulation point. For example $S=E$. In the
particular case where $S=\left( \varepsilon _{n}\right) _{n\in \mathbb{N}}$
with $\varepsilon _{n}\in E$, $\varepsilon _{n}\leq e$ and $\varepsilon
_{n}\rightarrow \theta $ as $n\rightarrow +\infty $, we will refer to $S$ as
a \textit{fundamental sequence.} Finally, let $1\leq p<+\infty $, and let $A$
be a given homogenization algebra on $\left( X,\mathcal{H},\lambda \right) $.

\begin{definition}
\emph{\label{def6.2}} \emph{A sequence }$\left( u_{\varepsilon }\right)
_{\varepsilon \in S}$\emph{\ in }$L^{p}\left( \Omega \right) $\emph{\ (}$%
\Omega $\emph{\ provided with the induced measure }$\lambda _{\Omega
}=\lambda \mathfrak{\mid }_{\Omega }$ \emph{if }$\Omega \neq X$\emph{) is
said to weakly }$\Sigma $\emph{-converge in }$L^{p}\left( \Omega \right) $%
\emph{\ to some }$u_{0}\in L^{p}\left( \Omega \times \Delta \left( A\right)
\right) $\emph{\ if as }$S\ni \varepsilon \rightarrow \theta $\emph{,}%
\begin{equation*}
\int_{\Omega }u_{\varepsilon }\left( x\right) \psi ^{\varepsilon }\left(
x\right) d\lambda \left( x\right) \rightarrow \int \int_{\Omega \times
\Delta \left( A\right) }u_{0}\left( x,s\right) \widehat{\psi }\left(
x,s\right) d\lambda \left( x\right) d\beta \left( s\right)
\end{equation*}%
\emph{for all }$\psi \in L^{p^{\prime }}\left( \Omega ;A\right) $\emph{\ (}$%
\frac{1}{p^{\prime }}=1-\frac{1}{p}$\emph{), where }$\psi ^{\varepsilon }$%
\emph{\ is defined as in (\ref{eq4.3}) (thanks to Lemma \ref{lem4.2}), and }$%
\widehat{\psi }$\emph{\ is the function in }$L^{p^{\prime }}\left( \Omega ;%
\mathcal{C}\left( \Delta \left( A\right) \right) \right) $\emph{\ given by} $%
\widehat{\psi }\left( x\right) =\mathcal{G}\left( \psi \left( x\right)
\right) $ \emph{(}$\mathcal{G}$ \emph{the Gelfand transformation on }$A$%
\emph{).}
\end{definition}

This notion generalizes the notion of the same name introduced earlier in
the framework of numerical spaces (see \cite{bib18}). It also generalizes
the well known concept of two-scale convergence (see \cite[17, 26]{bib1})
and further it can be shown that most of the results available in the
two-scale convergence setting carry over mutatis mutandis to the present
setting. We will here focus on the $\Sigma $-compactness theorem
generalizing the so-called two-scale compactness theorem (see \cite[1, 17]%
{bib19}).

\begin{theorem}
\label{th6.1} ($\Sigma $-compactness theorem). Assume that $1<p<+\infty $.
Suppose $X$ is metrizable. Let $S$ be a fundamental sequence and let a
sequence $\left( u_{\varepsilon }\right) _{\varepsilon \in S}$ be bounded in 
$L^{p}\left( \Omega \right) $. Then, a subsequence $S^{\prime }$ can be
extracted from $S$ such that the corresponding sequence $\left(
u_{\varepsilon }\right) _{\varepsilon \in S^{\prime }}$ is weakly $\Sigma $%
-convergent in $L^{p}\left( \Omega \right) $.
\end{theorem}

\begin{proof}
First of all we note that, because $X$ is metrizable, the topology on $X$ (a 
$\sigma $-compact locally compact space) has a countable basis (use \cite%
{bib6}, Chap. IX, p.21, Corollaire) and hence the same is true of $\Omega $.
Therefore, since $A$ is separable, it follows that $L^{p^{\prime }}\left(
\Omega ;A\right) $ ($\frac{1}{p^{\prime }}=1-\frac{1}{p}$) is separable.

With this in mind, let%
\begin{equation*}
\mathcal{L}_{\varepsilon }\left( \psi \right) =\int_{\Omega }u_{\varepsilon
}\left( x\right) \psi \left( x,H_{\varepsilon }\left( x\right) \right)
d\lambda \left( x\right) \qquad \left( \psi \in L^{p^{\prime }}\left( \Omega
;A\right) \right) \text{,}
\end{equation*}%
where $\varepsilon $ is freely fixed. In view of Lemma \ref{lem4.2}, this
yields a sequence $\left( \mathcal{L}_{\varepsilon }\right) _{\varepsilon
\in S}$ in $[L^{p^{\prime }}\left( \Omega ;A\right) ]^{\prime }$
(topological dual of $L^{p^{\prime }}\left( \Omega ;A\right) $) which is
bounded in the $[L^{p^{\prime }}\left( \Omega ;A\right) ]^{\prime }$-norm.
Hence, by a classical argument (partly based on the separability of the
Banach space $L^{p^{\prime }}\left( \Omega ;A\right) $) one can extract a
subsequence $S^{\prime }$ from $S$ such that $\mathcal{L}_{\varepsilon
}\rightarrow \mathcal{L}$ in $[L^{p^{\prime }}\left( \Omega ;A\right)
]^{\prime }$-weak $\ast $ as $S^{\prime }\ni \varepsilon \rightarrow \theta $%
, or equivalently such that as $S^{\prime }\ni \varepsilon \rightarrow
\theta $, we have $\mathcal{L}_{\varepsilon }\left( \psi \right) \rightarrow
\left\langle L,\widehat{\psi }\right\rangle $ for all $\psi \in L^{p^{\prime
}}\left( \Omega ;A\right) $ ($\widehat{\psi }$ defined above), where $L$ is
the continuous linear form on $L^{p^{\prime }}\left( \Omega ;\mathcal{C}%
\left( \Delta \left( A\right) \right) \right) $ defined by $\left\langle
L,\varphi \right\rangle =\left\langle \mathcal{L},\mathcal{G}^{-1}\circ
\varphi \right\rangle $ for $\varphi \in L^{p^{\prime }}\left( \Omega ;%
\mathcal{C}\left( \Delta \left( A\right) \right) \right) $, $\mathcal{G}%
^{-1} $ being the inverse Gelfand transformation on $A$. The next point is
to characterize the functional $L$. To this end let $\psi \in \mathcal{K}%
\left( \Omega ;A\right) $. It is clear that%
\begin{equation*}
\left\vert \mathcal{L}_{\varepsilon }\left( \psi \right) \right\vert \leq
c\left( \int_{\Omega }\left\vert \psi \left( x,H_{\varepsilon }\left(
x\right) \right) \right\vert ^{p^{\prime }}\mathcal{\chi }_{K}\left(
x\right) d\lambda \left( x\right) \right) ^{\frac{1}{p^{\prime }}}\quad
\left( \varepsilon \in S\right) \text{,}
\end{equation*}%
where $c=\sup_{\varepsilon \in S}\left\Vert u_{\varepsilon }\right\Vert
_{L^{p}\left( \Omega \right) }<+\infty $, $K$ is a compact set in $\Omega $
containing the support of $\psi $, and $\mathcal{\chi }_{K}$ is the
characteristic \ function of $K$ in $\Omega $. By letting $S^{\prime }\ni
\varepsilon \rightarrow \theta $ and applying \ Proposition \ref{pr4.2}
(with $u\left( x,y\right) =\left\vert \psi \left( x,y\right) \right\vert
^{p^{\prime }}$, $x\in \Omega $, $y\in X$) we get 
\begin{equation*}
\left\vert \left\langle L,\widehat{\psi }\right\rangle \right\vert \leq
c\left\Vert \widehat{\psi }\right\Vert _{L^{p^{\prime }}\left( \Omega \times
\Delta \left( A\right) \right) }
\end{equation*}%
for any $\psi \in \mathcal{K}\left( \Omega ;A\right) $. Thus, 
\begin{equation*}
\left\vert \left\langle L,\varphi \right\rangle \right\vert \leq c\left\Vert
\varphi \right\Vert _{L^{p^{\prime }}\left( \Omega \times \Delta \left(
A\right) \right) }
\end{equation*}%
for all $\varphi \in \mathcal{K}\left( \Omega ;\mathcal{C}\left( \Delta
\left( A\right) \right) \right) =\mathcal{K}\left( \Omega \times \Delta
\left( A\right) \right) $. By extension by continuity [based on the density
\ of $\mathcal{K}\left( \Omega ;\mathcal{C}\left( \Delta \left( A\right)
\right) \right) $ in $L^{p^{\prime }}\left( \Omega \times \Delta \left(
A\right) \right) =L^{p^{\prime }}\left( \Omega ;L^{p^{\prime }}\left( \Delta
\left( A\right) \right) \right) $] and use of the Riesz representation
theorem, it follows that%
\begin{equation*}
\left\langle L,\widehat{\psi }\right\rangle =\int \int_{\Omega \times \Delta
\left( A\right) }u_{0}\left( x,s\right) \widehat{\psi }\left( x,s\right)
d\lambda \left( x\right) d\beta \left( s\right)
\end{equation*}%
for all $\psi \in L^{p^{\prime }}\left( \Omega ;A\right) $, where $u_{0}\in
L^{p}\left( \Omega \times \Delta \left( A\right) \right) $. The theorem
follows thereby.
\end{proof}

Theorem \ref{th6.1} generalizes the two-scale compactness theorem in three
major\ directions: 1) the numerical spaces $\mathbb{R}^{n}$ on\ which
classical homogenization theory\ is framed are here replaced by more general
locally compact spaces; 2) the usual group action $\left( \varepsilon
,x\right) \rightarrow \frac{x}{\varepsilon }$ (see Example \ref{ex2.4})
underlying classical homogenization theory is here replaced by more general
absorptive continuous $\mathbb{R}$-group actions; 3) instead of the
classical periodic setting, we have here a general setting associated to the
notion of a homogenization algebra and covering a great variety of
behaviours such as the periodicity, the almost periodicity, and many more
besides.

It appears that the present work lays the foundation of the mathematical
framework that is needed to undertake a systematic study of homogenization
problems on manifolds, Lie groups included.

\end{document}